\documentclass[pdflatex,sn-mathphys-num]{sn-jnl}

\usepackage{graphicx}%
\usepackage{multirow}%
\usepackage{amsmath,amssymb,amsfonts}%
\usepackage{amsthm}%
\usepackage{mathrsfs}%
\usepackage[title]{appendix}%
\usepackage{xcolor}%
\usepackage{textcomp}%
\usepackage{manyfoot}%
\usepackage{booktabs}%
\usepackage{algorithm}%
\usepackage{algorithmicx}%
\usepackage{algpseudocode}%
\usepackage{listings}%
\usepackage{subcaption}%

\usepackage{accents}%

\theoremstyle{thmstyleone}%
\newtheorem{theorem}{Theorem}
\newtheorem{proposition}[theorem]{Proposition}%

\theoremstyle{thmstyletwo}%
\newtheorem{assumption}{Assumption}%

\theoremstyle{thmstylethree}%

\begin{document}

\title{An Efficient Spatial Branch-and-Bound Algorithm for Global Optimization of Gaussian Process Posterior Mean Functions}


\author[1,2]{\fnm{Wei-Ting} \sur{Tang}}

\author[2]{\fnm{Akshay} \sur{Kudva}}

\author[3]{\fnm{Calvin} \sur{Tsay}}

\author*[1,2]{\fnm{Joel A.} \sur{Paulson}}\email{joel.paulson@wisc.edu}

\affil*[1]{\orgdiv{Department of Chemical and Biological Engineering}, \orgname{University of Wisconsin-Madison}, \orgaddress{\street{1415 Engineering Drive}, \city{Madison}, \postcode{53706}, \state{Wisconsin}, \country{USA}}}

\affil*[2]{\orgdiv{Department of Chemical and Biomolecular Engineering}, \orgname{The Ohio State University}, \orgaddress{\street{151 W. Woodruf Avenue}, \city{Columbus}, \postcode{43210}, \state{Ohio}, \country{USA}}}

\affil[3]{\orgdiv{Department of Computing}, \orgname{Imperial College London}, \orgaddress{\street{Exhibition Rd}, \city{South Kensington}, \postcode{SW7 2AZ}, \state{London}, \country{United Kingdom}}}


\abstract{We study the deterministic global optimization of trained Gaussian process posterior mean functions over hyperrectangular domains. Although the posterior mean function has a compact closed-form representation, its global optimization is challenging because it remains nonlinear and nonconvex. Existing exact deterministic approaches become increasingly difficult to scale as the number of training data points grows, leading to approximation-based methods that improve tractability by optimizing a modified (inexact) objective. In this work, we propose PALM-Mean, a piecewise-analytic lower-bounding framework embedded in reduced-space spatial branch-and-bound. At each node, kernel terms that are locally important are replaced by a sign-aware piecewise-linear relaxation in an appropriate scalar distance variable, while the remaining terms are bounded analytically in closed form. We show this hybrid approach yields a valid lower bound for the posterior mean, while limiting the size of the branch-and-bound subproblems. We establish validity of the node lower bounds and $\varepsilon$-global convergence of the resulting algorithm. Computational results on synthetic benchmarks and real-world application problems show that PALM-Mean improves scalability relative to representative general-purpose deterministic global solvers, particularly as the number of training data points increases.}




\maketitle

\section{Introduction}
\label{sec:intro}

Trained machine learning (ML) surrogates are increasingly used not only to predict system responses, but also as algebraic components inside optimization models~\cite{misener2023formulating}. In many scientific and engineering applications, the expensive step is generating data from simulations or experiments, whereas the downstream task is to repeatedly optimize a learned predictor within design, calibration, control, or experimental-planning workflows. This viewpoint fits naturally within the broader literature on computer experiments and surrogate-based analysis and optimization, where the central objective is to replace an expensive model by a fast emulator that can then be analyzed and optimized directly \cite{Sacks1989DACE,SantnerWilliamsNotz2003,SantnerWilliamsNotz2019,Queipo2005Surrogate,ForresterKeane2009}.

Gaussian processes (GPs), which are closely related to Kriging models in the computer-experiments literature \cite{kleijnen2009kriging}, remain especially important in this setting. Their main appeal is that they provide a flexible non-parametric regression model together with a principled quantification of predictive uncertainty, which is one reason they continue to play a central role in sample-efficient design and Bayesian optimization (BO) \cite{ref:gp_book,ref:bo_tutorial,ref:ego,ref:bo_review_a,ref:bo_review_b,paulson2025bayesian}. At the same time, once a GP has been trained, one is often left with a surrogate model that can be embedded in decision-making problems~\cite{ref:gp_embedded_exact}. Specifically, we consider the deterministic optimization problem over a fixed surrogate model rather than a sequential learning problem, which is an important distinction in this work. 

A trained GP with inputs $X=[x_1,\ldots,x_N]^\top$, observations $y \in \mathbb{R}^N$, and a stationary kernel $k_L(\cdot,\cdot)$ with automatic relevance determination (ARD) (diagonal) lengthscale matrix $L$ defines a posterior distribution over functions. In this setting, the posterior mean may be written as
\begin{equation}
\mu(x)=\sum_{i=1}^N \alpha_i k_L(x,x_i), \qquad 
\alpha = \left(K_{XX}+\sigma_\epsilon^2 I\right)^{-1} y.
\label{eq:gp_mean_sum}
\end{equation}

Although simple in appearance, global optimization over this mean function remains challenging due to its nonlinear, multi-modal nature. 
In this paper, we study the deterministic global minimization of $\mu(x)$ over a compact hyperrectangular domain. This problem arises, for example, when a GP-based workflow ultimately selects a recommendation by minimizing the trained posterior mean, but it also appears more broadly whenever a trained GP mean is embedded as a surrogate objective inside a larger optimization model \cite{Sacks1989DACE,SantnerWilliamsNotz2003,SantnerWilliamsNotz2019,Queipo2005Surrogate,ForresterKeane2009,Jones2001Taxonomy}. Although $\mu(x)$ is available in closed form, it is generally a non-convex weighted sum of kernel evaluations centered at all training points. Mixed signs of the entries in $\alpha$ and the direct dependence on every training point make global optimization nontrivial, especially as $N$ increases.

From the standpoint of deterministic global optimization, one straightforward route is to embed the GP equations in a full-space nonlinear or mixed-integer nonlinear optimization model and pass that formulation to a general-purpose global solver such as BARON \cite{ref:baron} or ANTIGONE \cite{ref:antigone}. That strategy is generic and useful, but it does not necessarily leverage the structure of the posterior mean. More broadly, deterministic global optimization has long relied on methods such as spatial branch-and-bound \cite{falk1969algorithm}, reduced-space formulations, relaxation-based bounding, and bound-tightening ideas to obtain certified global solutions of non-convex problems \cite{HorstTuy1996Global,RyooSahinidis1995Global,BelottiEtAl2009BT,BelottiEtAl2013MINLP}. These ideas provide the right algorithmic language for the present problem as well.

Reduced-space viewpoints are especially attractive in this setting because branching is only performed in the original decision variables, while lower bounds are propagated through the surrogate expression. For trained GPs, the most relevant exact benchmark is the reduced-space deterministic framework of \cite{ref:gp_embedded_exact}, which derives custom McCormick-based relaxations \cite{ref:reduced_space_bb,ref:deterministic_global_opt_text} directly from the explicit GP equations. As stated previously, this line of work targets the exact posterior mean, but it was found to scale poorly with number of training points, as the node lower-bounding problem must still account for one kernel contribution per data point. A different line of work improves tractability by instead approximating the kernel function in an optimization-friendly way, and then solving the resulting mixed-integer model \cite{ref:pwl_kernel_aistats}. That direction is closely related to a broader literature on disjunctive and piecewise-linear mixed-integer formulations \cite{VielmaNemhauser2011Disjunctive,HuchetteVielma2023Nonconvex,GrimstadKnudsen2020PWP}, but from the standpoint of posterior mean minimization it changes the problem being solved (i.e., the optimizer is globally optimal for the approximation, not necessarily for the original posterior mean). 

This leaves a natural gap in the literature. Existing exact deterministic approaches preserve fidelity to the trained posterior mean, but they can become increasingly difficult to scale as the problem size grows. Approximation-based methods can improve tractability, but they do so by replacing the objective. In this paper, we address this methodological gap for posterior mean minimization. We introduce PALM-Mean, a piecewise-analytic lower-bounding framework embedded in reduced-space spatial branch-and-bound. At each node, we define kernel terms that are `locally important' and treat them with a novel sign-aware piecewise-linear relaxation in an appropriate scalar distance variable, while the remaining terms are bounded analytically in closed form. Our goal is to preserve deterministic validity for the exact posterior mean while allocating computational effort selectively rather than uniformly across all kernel terms. 
More generally, this work contributes to the emerging literature on developing tailored optimization formulations for trained machine learning surrogates~\cite{misener2023formulating,ref:mlopt_review}, where most advances have focused on neural network models~\cite{anderson2020strong,grimstad2019relu,huchette2026deep,tsay2021partition}, leaving GP surrogates comparatively understudied despite their above-mentioned advantages.

The main contributions of this paper are threefold. First, we derive a valid hybrid lower-bounding procedure for GP posterior mean minimization that combines tighter piecewise relaxations on selected kernel terms with inexpensive analytical bounds on the remaining terms. Second, we introduce a specialized spatial branch-and-bound algorithm based on this construction for the exact global solution of the original posterior mean problem. Third, we provide a comprehensive computational study that clarifies when and why the proposed method improves scalability relative to alternative exact reduced-space approaches. The remainder of the paper is organized as follows. Section~\ref{sec:prob-statement} introduces the problem and notation. Section~\ref{sec:methodology} presents the hybrid lower-bounding strategy and associated branch-and-bound method. Section~\ref{sec:numerical-experiments} reports computational results and Section~\ref{sec:conclusions} provides some concluding remarks.

\section{Problem statement and preliminaries}
\label{sec:prob-statement}

\subsection{Gaussian process posterior mean}

Consider a Gaussian process regression model~\cite{ref:gp_book} trained on the dataset $\mathcal{D}=\{X,y\}$, where the training inputs are
\[
X=
\begin{bmatrix}
x_1^\top\\
\vdots\\
x_N^\top
\end{bmatrix}
\in\mathbb{R}^{N\times D},
\qquad
x_i\in\mathbb{R}^D,
\]
and the corresponding observations are $y=[y_1,\ldots,y_N]^\top\in\mathbb{R}^N$. We assume the standard observation model with Gaussian measurement noise
\[
y_i=f(x_i)+\varepsilon_i,
\qquad
\varepsilon_i\sim\mathcal{N}(0,\sigma_\epsilon^2),
\qquad i=1,\ldots,N,
\]
together with a zero prior mean after centering the response data. Under these assumptions, the posterior mean may be written in the compact form \eqref{eq:gp_mean_sum} (which is a literature-standard assumption \cite{bradford2018efficient, paulson2022cobalt}). This representation is central to our following derivations, since it explicitly shows that the trained posterior mean is a weighted sum of $N$ individual kernel term -- one for each training point -- with coefficients $\alpha_i\in\mathbb{R}$ that may have mixed signs (meaning they can be either positive or negative).

\subsection{Optimization problem}

For clarity, we consider the deterministic global minimization of the trained posterior mean over a compact domain:
\begin{equation}
\min_{x\in\mathcal{X}} \ \mu(x).
\label{eq:mean_problem}
\end{equation}
Nevertheless, our derivations are general, and the proposed methodologies can be applied to more general optimization problems involving $\mu(x)$. 
The standing assumptions used throughout the paper are summarized next.

\begin{assumption}
\label{assump:standing}
The following conditions hold throughout this paper.
\begin{enumerate}
    \item The response data have been centered such that the GP prior mean is zero and the posterior mean is given by \eqref{eq:gp_mean_sum}.
    \item The covariance kernel is a stationary function of distance and monotonically nonincreasing with respect to scaled distance.
    \item The feasible region is a hyperrectangle
    \[
    \mathcal{X}=[x^L,x^U]:=\{x\in\mathbb{R}^D: x^L\le x\le x^U\},
    \]
    where $x^L,x^U\in\mathbb{R}^D$ and $x^L<x^U$ componentwise.
    \item The training inputs $X$, observations $y$, noise variance $\sigma_\epsilon^2$, and kernel hyperparameters are fixed during optimization (i.e., the surrogate model is trained).
\end{enumerate}
\end{assumption}

These assumptions isolate the global optimization problem for a trained GP from the separate issues of model fitting and hyperparameter estimation \cite{ref:gp_book,ref:gp_regression_review}. In particular, we do not optimize over kernel hyperparameters or marginal log-likelihood expressions. The object of interest is the trained posterior mean itself.

Under Assumption~\ref{assump:standing}, the function $\mu(\cdot)$ is continuous on the compact set $\mathcal{X}$, so problem \eqref{eq:mean_problem} admits at least one global minimizer. For the spatial branch-and-bound framework developed later, we use $X_h\subseteq \mathcal{X}$ to denote a generic node in the partition tree. Each node $X_h$ is again a hyperrectangle, and the main algorithmic task is to compute a valid lower bound on $\min_{x\in X_h}\mu(x)$.

\subsection{Kernel classes and ARD notation}

We focus on common stationary kernels with automatic relevance determination (ARD) \cite{ref:gp_book}. We denote the diagonal matrix of lengthscales
\[
L=\mathrm{diag}(\ell_1,\ldots,\ell_D),
\qquad
\ell_j>0 \ \ \forall j=1,\ldots,D,
\]
and define the scaled distance from $x$ to training point $x_i$ by
\begin{equation}
r_i(x):=\|L^{-1}(x-x_i)\|_2,
\qquad
d_i(x):=r_i(x)^2.
\label{eq:scaled_distance}
\end{equation}
The stationary kernels considered in this paper depend only on this scaled separation. We specifically consider the squared exponential kernel, also referred to as the radial basis function (RBF) kernel, together with the Mat\'ern-$\nu$ family for $\nu\in\{1/2,3/2,5/2\}$, noting their broad popularity in modeling applications~\cite{ref:gp_book}:
\[
k_{\mathrm{RBF}}(x,x_i)=\sigma_f^2\exp\!\left(-\frac{1}{2}d_i(x)\right),
\]
\[
k_{1/2}(x,x_i)=\sigma_f^2 e^{-r_i(x)},
\]
\[
k_{3/2}(x,x_i)=\sigma_f^2\left(1+\sqrt{3}\,r_i(x)\right)e^{-\sqrt{3}\,r_i(x)},
\]
\[
k_{5/2}(x,x_i)=\sigma_f^2\left(1+\sqrt{5}\,r_i(x)+\frac{5}{3}r_i(x)^2\right)e^{-\sqrt{5}\,r_i(x)},
\]
where $\sigma_f^2 > 0$ denotes the scale factor that controls the overall amplitude (i.e., the vertical spread) of functions in the posterior distribution. For the RBF kernel, the natural scalar argument is the squared scaled distance $d_i(x) = r_i(x)^2$, whereas for the Mat\'ern kernels it is the scaled distance $r_i(x)$. This distinction matters for lower-bounding constructions, even though the posterior mean always retains the form \eqref{eq:gp_mean_sum}.

\subsection{Relevant prior work}

Problem \eqref{eq:mean_problem} is a smooth but generally non-convex nonlinear optimization problem. Although one may always write the GP equations in a full-space formulation and call a general-purpose deterministic global solver, such as BARON \cite{ref:baron} or ANTIGONE \cite{ref:antigone}, that approach does not automatically exploit the structure of the posterior mean and may introduce many auxiliary variables tied to the kernel evaluations. This issue becomes more pronounced as the number of training points increases.

For that reason, ``reduced-space formulations'' are especially attractive in this setting. In reduced space, branching is performed directly in the original decision variables, while lower bounds are propagated through the predictor expression itself \cite{ref:reduced_space_bb,ref:deterministic_global_opt_text}. The most relevant deterministic benchmark here is the reduced-space approach of Schweidtmann et al., who propose a spatial branch-and-bound framework for trained GP models using McCormick-based relaxations of the explicit GP equations \cite{ref:gp_embedded_exact}. Their method targets the exact GP model rather than an approximation. It also draws on the broader machinery of factorable programming and McCormick relaxations, which remain central tools in deterministic global optimization \cite{ref:mccormick,ref:mccormick_algorithms,ref:global_opt_text_a,ref:global_opt_text_b}.

At the same time, the structure of \eqref{eq:gp_mean_sum} creates a clear computational difficulty. Even if one can derive strong relaxations for each individual term $\alpha_i k(x,x_i)$ over a node $X_h$, the lower bound for the posterior mean must still aggregate all $N$ contributions. As $N$ grows, the relaxation quality can deteriorate because the node bound is formed from a composition over many kernel terms with coefficients that may have different signs. In practice, this can lead to substantial cancellation and weaker lower bounds, which in turn increases the amount of partitioning required.

A different line of work improves tractability by replacing the exact kernel profile with a piecewise-linear approximation, yielding mixed-integer formulations that can be computationally effective in practice \cite{ref:pwl_kernel_aistats}. This direction is promising, but it inherently changes the problem being solved. The resulting optimizer is globally optimal for the approximating model, not necessarily for the original posterior mean. Though the authors provide some suboptimality bounds for the BO context, this leaves an opening for methods that preserve deterministic validity for the exact objective while scaling more effectively than existing approaches as the posterior sum grows (which is exactly what we pursue in this paper).

A related line of work focuses on the global optimization of common acquisition functions in BO, which are themselves constructed from GP posteriors, e.g., Thompson sampling (TS) and lower confidence bound (LCB). \citet{adebiyi2024optimizing} propose an initial sampling strategy based on root-finding to improve solution quality of gradient-based local solvers when optimizing the TS.  \citet{georgiou2025deterministic} employ the deterministic global solver MAiNGO \cite{ref:maingo} within the inner acquisition (LCB) optimization loop of BO. The authors analyze the effect of local/global solutions of LCB on BO's performance. Different from these prior works, our paper focuses on the global minimization of GP posterior mean function in the surrogate modeling context.

\section{Methodology}
\label{sec:methodology}

We now introduce PALM-Mean, a piecewise-analytic lower-bounding method and a corresponding spatial branch-and-bound framework for the exact global solution of \eqref{eq:mean_problem}. The key idea is to construct, on each branch-and-bound node $X_h$, a valid lower-bounding problem that is substantially tighter than a purely analytical interval bound, while avoiding the cost of piecewise treatment for every kernel term. The main text below in this section presents the core construction, while more explicit/detailed formulas and pseudocode are provided in the appendix. 

\subsection{Spatial branch-and-bound overview}

Let $X_h\subseteq\mathcal{X}$ denote a subdomain corresponding to a node in the branch-and-bound tree, with $X_h=[x_h^L,x_h^U]$. A valid lower bound on this node is any scalar $\mathrm{LB}(X_h)$ satisfying
\[
\mathrm{LB}(X_h)\le \min_{x\in X_h}\mu(x).
\]
If $\mathrm{UB}^\star$ denotes the best objective value currently available from feasible points, then node $X_h$ can be discarded, or `fathomed,' whenever $\mathrm{LB}(X_h)\ge \mathrm{UB}^\star-\varepsilon$ for the chosen optimality tolerance $\varepsilon>0$. Thus, the lower-bounding problem must do two things well. It must be \emph{valid}, so that no globally optimal point is incorrectly pruned, and it must be sufficiently \emph{tight and inexpensive} so that the branch-and-bound procedure does not become computationally intractable.

PALM-Mean works in reduced space, meaning branching is performed only in the original decision vector $x$, while auxiliary variables are introduced only as needed inside the node lower-bounding problem. On each branch-and-bound node with subdomain $X_h$, the algorithm solves a lower-bounding problem, updates the incumbent by local optimization of the exact posterior mean, and then branches on one coordinate of $X_h$ if the node is not fathomed. Since the upper-bounding step uses the exact objective and the lower-bounding step is valid by construction, convergence to a global optimum follows from standard spatial branch-and-bound arguments \cite{ref:reduced_space_bb,ref:deterministic_global_opt_text}.

\subsection{Sign-aware bounds for individual kernel terms}

The posterior mean contains coefficients of mixed sign, so a valid lower bound cannot be constructed from kernel lower bounds alone. 
In other words, the uncertainty of the sign associated with each kernel-function term dictates that both lower and upper bounds be considered. 
This is the basic reason that sign-aware envelopes are required. For convenience, define
\[
\alpha_i^+ := \max\{0,\alpha_i\},
\qquad
\alpha_i^- := \max\{0,-\alpha_i\},
\qquad i=1,\ldots,N,
\]
so that $\alpha_i=\alpha_i^+-\alpha_i^-$. We also define the shorthand $k_i(x):=k(x,x_i)$. We can then introduce the following proposition to create upper and lower estimators of $\mu(x)$. 

\begin{proposition}
\label{prop:sign_aware_term_bounds}
Suppose that, for some node subdomain $X_h$, the functions $\underline{k}_{i,h}$ and $\overline{k}_{i,h}$ satisfy
\[
\underline{k}_{i,h}(x)\le k_i(x)\le \overline{k}_{i,h}(x),
\qquad \forall x\in X_h.
\]
Then
\[
\underline{t}_{i,h}(x):=\alpha_i^+\underline{k}_{i,h}(x)-\alpha_i^-\overline{k}_{i,h}(x)
\]
and
\[
\overline{t}_{i,h}(x):=\alpha_i^+\overline{k}_{i,h}(x)-\alpha_i^-\underline{k}_{i,h}(x)
\]
satisfy
\[
\underline{t}_{i,h}(x)\le \alpha_i k_i(x)\le \overline{t}_{i,h}(x),
\qquad \forall x\in X_h.
\]
Consequently, $\sum_{i=1}^N \underline{t}_{i,h}(x)$ is a valid pointwise lower bound of $\mu(x)$ on $X_h$.
\end{proposition}

\begin{proof}
Take arbitrary index $i$ and point $x\in X_h$. If $\alpha_i\ge0$, multiplying the kernel bounds by $\alpha_i$ preserves the inequality direction and yields
\[
\alpha_i\underline{k}_{i,h}(x)\le \alpha_i k_i(x)\le \alpha_i\overline{k}_{i,h}(x).
\]
If $\alpha_i<0$, multiplying by $\alpha_i$ reverses the inequality direction and yields
\[
\alpha_i\overline{k}_{i,h}(x)\le \alpha_i k_i(x)\le \alpha_i\underline{k}_{i,h}(x).
\]
Writing $\alpha_i=\alpha_i^+-\alpha_i^-$ combines these two cases into the sign-aware expressions stated in the proposition. Since the argument is valid for each $i$ separately, summing over $i=1,\ldots,N$ gives a pointwise lower bound of the full posterior mean over domain $X_h$.
\end{proof}

Proposition~\ref{prop:sign_aware_term_bounds} reduces the lower-bounding task over a node subdomain to the construction of valid lower and upper bounds for each individual kernel term on $X_h$. We consider two complementary constructions below. The first is an $x$-dependent piecewise-linear envelope that is tighter but more expensive. The second is a closed-form analytical box bound that is much cheaper but typically weaker.

For the RBF kernel, it is natural to work in the scalar variable $d_i(x)$ from \eqref{eq:scaled_distance}, since the kernel profile $\kappa_{\mathrm{RBF}}(d)=\sigma_f^2 e^{-d/2}$ is decreasing and convex on $[0,\infty)$. Tangent lines are therefore valid lower bounds, and secant lines are valid upper bounds on any interval in $d$. For the Mat\'ern-$1/2$ kernel, the same rule holds in the scalar variable $r_i(x)$. For the Mat\'ern-$3/2$ and Mat\'ern-$5/2$ kernels, the scalar profiles remain decreasing, but each has a single inflection point. As a consequence, the tangent/secant roles as lower and upper bounds depend on whether a segment lies in a concave or convex region. Exact formulas, curvature thresholds, and segmentwise envelope definitions are given in Appendix~\ref{app:tangent_secant}.

Figure~\ref{fig:kernel_relaxation} illustrates why this distance-based piecewise construction is attractive. In the scalar distance variable, the piecewise-linear relaxation tracks the RBF kernel closely for both positive and negative term contributions. Although the McCormick relaxation of a kernel with positive contribution is exact on the original function in the scalar distance variable, the McCormick relaxation of a kernel with negative contribution in the scalar distance variable can be much weaker than the piecewise-linear relaxation. Furthermore, the McCormick relaxation builds the bound of the kernel function in the original space $x$ through the propagation of McCormick relaxations through factorable functions \cite{ref:mccormick}. As shown in the bottom row, the final bound of McCormick relaxation in the original variable space $x$ is much weaker than the piecewise linear relaxation. This looseness matters because lower-bound quality at the level of individual kernel terms propagates directly into the node lower bound for the full posterior mean. Furthermore, given the nonlinearity of the McCormick-based bounds, a deterministic solver employing them (e.g. MAiNGO \cite{ref:maingo}) typically performs linearization at the midpoint to formulate the lower bounding problem as a linear program \cite{bongartz2017deterministic}, which can further enlarge the gap between the lower bound and the true kernel function.


\begin{figure}[t]
     \centering
     \begin{subfigure}[b]{0.9\textwidth}
         \centering
         \includegraphics[width=\textwidth]{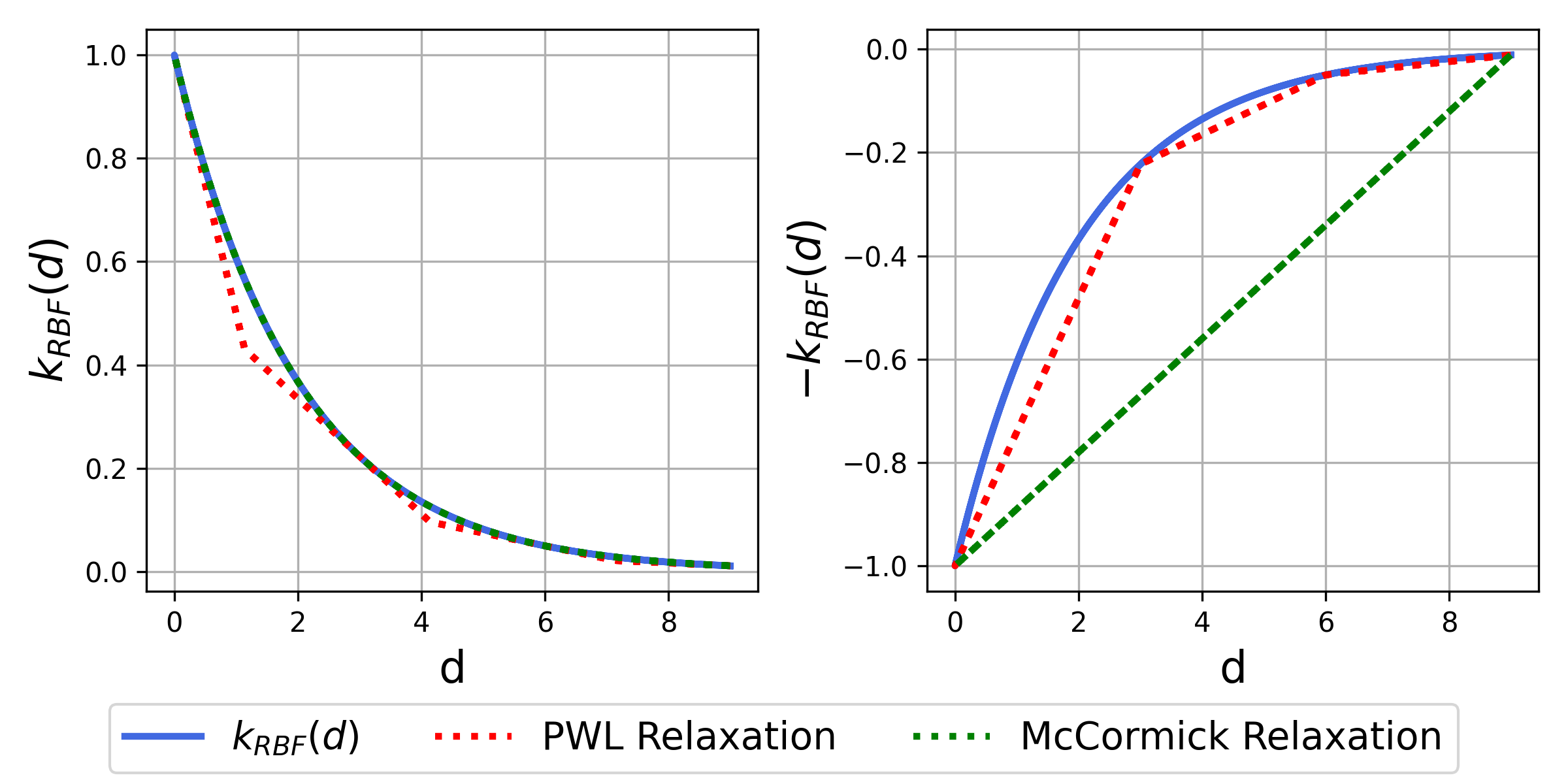}
     \end{subfigure}
     
     \begin{subfigure}[b]{0.9\textwidth}
         \centering
         \includegraphics[width=\textwidth]{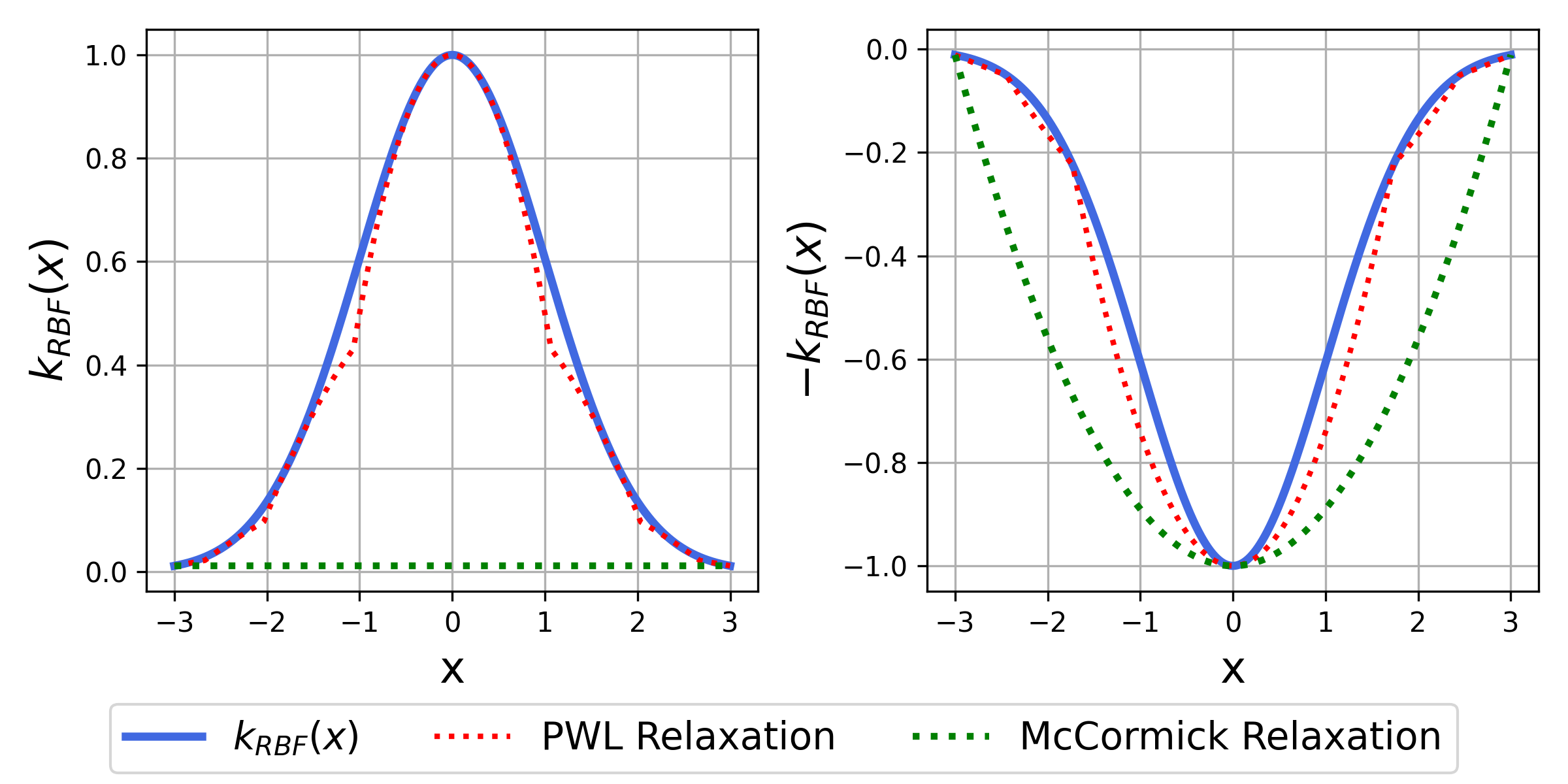}
     \end{subfigure}
     \caption{Illustration of piecewise-linear and McCormick relaxation for the RBF kernel. The top row shows sign-aware envelopes in scalar distance variable $d$. The McCormick relaxation overlaps with the original kernel function with positive contribution, which is the tightest possible convex relaxation. The distance-based piecewise relaxation is substantially tighter than McCormick relaxation for kernel with negative contribution. The bottom row compares the resulting lower bounds in the original variable space, where the piecewise relaxation is much tighter than McCormick relaxation in both positive and negative contribution cases.}
\label{fig:kernel_relaxation}
\end{figure}

As mentioned above, as a cheaper alternative to these piecewise-linear bounds, we also introduce a closed-form analytical bound over a hyperrectangle. Since all kernels considered here are stationary and monotonically decreasing in scaled distance, analytical kernel bounds reduce to finding the minimum and maximum scaled distance from $x_i$ to the node box $X_h$. These distances are available in closed form from the coordinate-wise projection of $x_i$ onto $X_h$ and from the furthest vertex of $X_h$. The resulting analytical bounds are detailed in Appendix~\ref{app:analytic_bounds}. Applying Proposition~\ref{prop:sign_aware_term_bounds} then yields inexpensive constant lower and upper bounds for each kernel contribution on $X_h$, which can immediately be composed to bounds over $\mu(x)$.

\subsection{Important and unimportant terms}
\label{sec:important_terms}
The key observation behind PALM-Mean is that, for a given node subdomain $X_h$, only a subset of kernel terms typically has enough potential influence to justify the tighter, but more expensive piecewise-linear bounding treatment. The remaining terms can be bounded analytically at much lower cost.

For each training point $x_i$, define the nodewise contribution score as follows
\begin{equation}
\beta_{i,h}:=\max_{x\in X_h} |\alpha_i k_i(x)|.
\label{eq:beta_def}
\end{equation}
Since the kernels are non-negative and monotonically decreasing in scaled distance, the analytical upper bounds derived in Appendix~\ref{app:analytic_bounds} imply that $\beta_{i,h}=|\alpha_i|\,\overline{k}_{i,h}^{A}$.
Here, $\underline{k}_{i,h}^{A}$ and $\overline{k}_{i,h}^{A}$ denote the analytical lower and upper bounds, respectively, on the kernel term $k_i(x)$ over node $X_h$, while
\[
\underline{t}_{i,h}^{A}
:=
\alpha_i^+ \underline{k}_{i,h}^{A}
-
\alpha_i^- \overline{k}_{i,h}^{A},
\qquad
\overline{t}_{i,h}^{A}
:=
\alpha_i^+ \overline{k}_{i,h}^{A}
-
\alpha_i^- \underline{k}_{i,h}^{A}
\]
denote the corresponding analytical lower and upper bounds on the signed contribution $\alpha_i k_i(x)$. Thus, $\beta_{i,h}$ is available in closed form. Let
$\rho\in[0,1]$ be a specified heuristic importance threshold. We define the important and unimportant index sets on node $X_h$ by
\[
\mathcal{I}(h):=\{i:\beta_{i,h}\ge \rho\},
\qquad
\mathcal{U}(h):=\{1,\ldots,N\}\setminus \mathcal{I}(h).
\]
When $\rho=0$, every term is treated as important and the method reduces to a fully piecewise-linear node formulation, i.e., constructing piecewise-linear bounding functions for every kernel term. As $\rho$ increases, fewer terms receive the full piecewise-linear treatment. 
In other words, the choice of $\rho$ balances the tradeoff between overall bound tightness and computational expense. In this work. we set $\rho=0.01$.

Our rationale is straightforward: if $\beta_{i,h}$ is small, then even the largest possible magnitude of term $\alpha_i k_i(x)$ on $X_h$ is small relative to the dominant terms within that node subdomain. In such cases, formulating piecewise-linear bounding functions (e.g., using binary variables or special ordered sets) for term $i$ often yields limited practical benefit, which we summarize in the next proposition. 

\begin{proposition}
\label{prop:analytic_gap_beta}
For each index $i \in \{1,..,D\}$ and subdomain $X_h$, the analytically derived bounds satisfy
\[
\overline{t}_{i,h}^{A}-\underline{t}_{i,h}^{A}
=
|\alpha_i|\left(\overline{k}_{i,h}^{A}-\underline{k}_{i,h}^{A}\right)
\le \beta_{i,h}.
\]
Consequently, if $i\in\mathcal{U}(h)$, then
\[
\overline{t}_{i,h}^{A}-\underline{t}_{i,h}^{A}
\le \rho.
\]
\end{proposition}

\begin{proof}
By definition of the sign-aware analytical bounds,
\[
\overline{t}_{i,h}^{A}-\underline{t}_{i,h}^{A}
=
(\alpha_i^++\alpha_i^-)\left(\overline{k}_{i,h}^{A}-\underline{k}_{i,h}^{A}\right)
=
|\alpha_i|\left(\overline{k}_{i,h}^{A}-\underline{k}_{i,h}^{A}\right).
\]
Since $0\le \underline{k}_{i,h}^{A}\le \overline{k}_{i,h}^{A}$, the right-hand side is bounded from above by $|\alpha_i|\overline{k}_{i,h}^{A}=\beta_{i,h}$. If $i\in\mathcal{U}(h)$, then $\beta_{i,h}<\rho$ by definition, and the second claim follows.
\end{proof}

Proposition~\ref{prop:analytic_gap_beta} does not imply that unimportant terms are negligible when taken in composition.  Rather, it shows that each unimportant (and hence treated with the analytical relaxation) term has limited individual range on $X_h$ relative to the dominant terms. PALM-Mean therefore allocates the more expensive piecewise-linear bounds only to kernel functions at indices in $\mathcal{I}(h)$, while kernels with indices in $\mathcal{U}(h)$ are handled with relatively inexpensive analytical bounds.

Figure~\ref{fig:analytic_lb}(a) illustrates this overall idea. Over a fixed spatial region, the local shape of the posterior mean is determined primarily by kernel function to data points near that region, while data points further away contribute less influence over the node subdomain. 
Intuitively, the GP surrogate model implicitly encodes that nearby observations are more informative than distant ones, with the kernel function assigning higher correlation (and thus greater influence) to points that are close in the input space. 
This is exactly the structure exploited by the important/unimportant partition.

\begin{figure}[t]
\centering
\includegraphics[width=\textwidth]{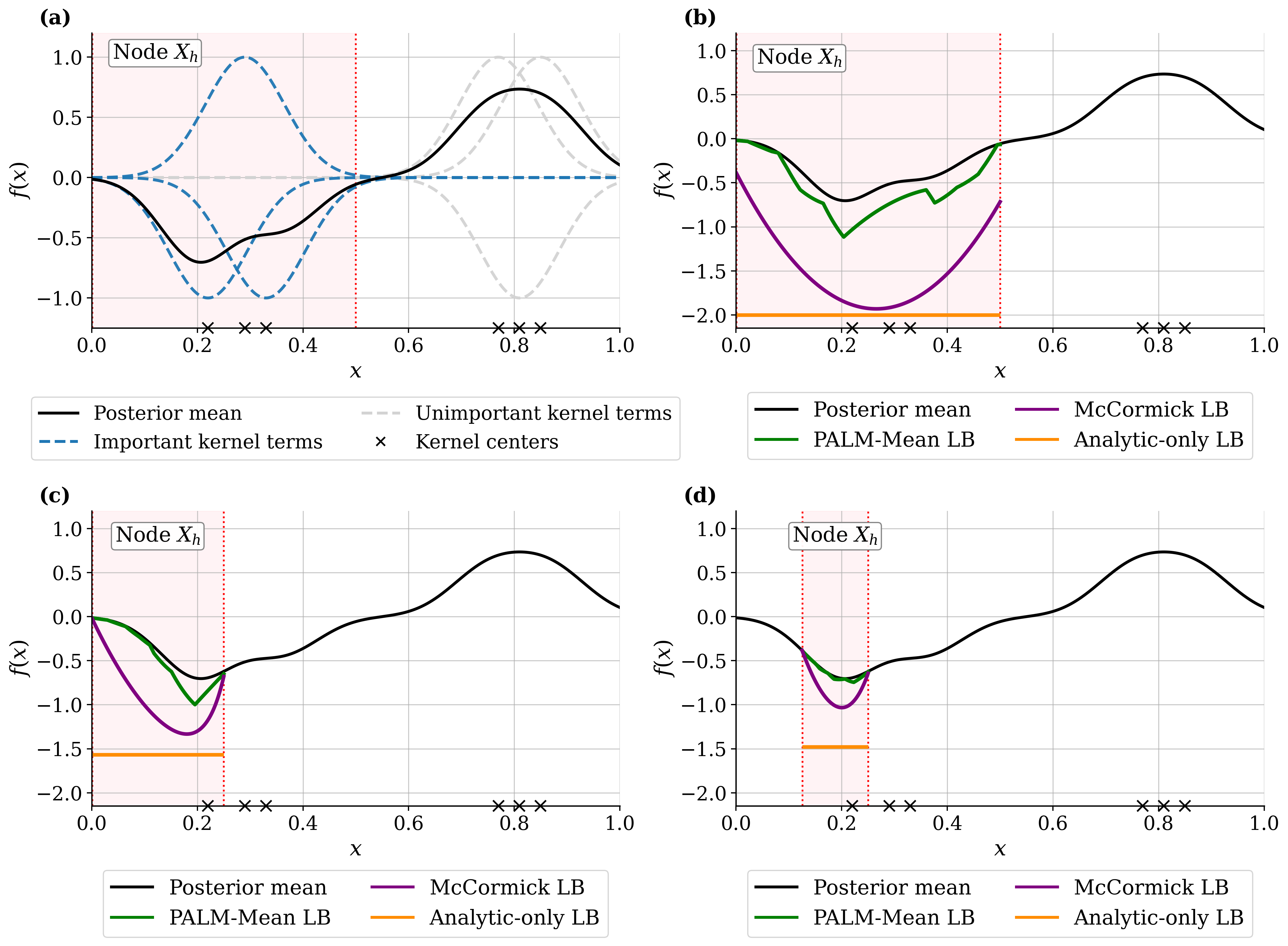}
\caption{Illustrative figure for PALM-Mean. (a) Illustration of local term influence in a one-dimensional posterior mean. Kernel with centers close to the current spatial region (dashed blue curve) determine the local shape of the mean within the node, while kernel with distant centers contribute little variation over the node (dashed light grey curve). This motivates analytical treatment of unimportant terms and piecewise treatment of important terms. (b)-(d) Iteration 1-3 for spatial branch-and-bound algorithm with hybrid node lower bound of PALM-Mean (green curve), McCormick relaxation (purple curve), and analytical bound (orange line). PALM-Mean's hybrid bound is the tightest, followed by McCormick relaxation and then analytical bound. Area shaded with light red denotes the active node $X_h$. We define 5 break points on each piecewise-linear bound.}
\label{fig:analytic_lb}
\end{figure}

\subsection{Hybrid lower bound}

Following the above motivation, we combine the piecewise-linear bounds for important terms (defined as in Section~\ref{sec:important_terms}) with the analytical bounds for remaining unimportant terms to yield the PALM-Mean lower bound
\begin{equation}
\underline{\mu}_h^{\mathrm{PALM}}(x)
:=
\sum_{i\in\mathcal{I}(h)}
\left(
\alpha_i^+ \underline{k}_{i,h}^{P}(x)
-
\alpha_i^- \overline{k}_{i,h}^{P}(x)
\right)
+
\sum_{i\in\mathcal{U}(h)}
\underline{t}_{i,h}^{A}.
\label{eq:PALM_functional_lb}
\end{equation}
The associated lower bound to be used in branch-and-bound nodes is then
\begin{align}
\mathrm{LB}^{\mathrm{PALM}}(X_h)
=
\min_{x\in X_h}\underline{\mu}_h^{\mathrm{PALM}}(x).
\label{eq:PALM-lb}
\end{align}
A mixed-integer representation of this problem is obtained by introducing one scalar distance variable per important term together with piecewise-linear envelope variables. The resulting node problem is a (non-convex) mixed-integer quadratically constrained program (MIQCP). The mixed-integer structure enters through the segment-activation constraints, e.g., through special ordered sets or alternative formulation~\cite{HuchetteVielma2023Nonconvex}, while the non-convexity enters through the exact quadratic equalities linking the scalar distance variables to $x$. One such explicit encoding of the MIQCP is provided in Appendix~\ref{app:mip_encoding}.

Upper bounds during branch and bound are obtained separately by local optimization of the true posterior mean function $\mu(x)$ over $X_h$. If $\hat{x}_h\in X_h$ is any feasible point returned by a local solver, then $\mu(\hat{x}_h)$ is a valid upper bound on $\min_{x\in X_h}\mu(x)$.

\subsection{Validity, convergence, and implementation summary}

We now state the two basic properties required by the overall algorithm: (i) validity of the proposed hybrid lower bound over node subdomains and (ii) $\varepsilon$-global convergence of the resulting spatial branch-and-bound method.

\begin{proposition}
\label{thm:PALM_validity}
For every node subdomain $X_h\subseteq\mathcal{X}$, $\mathrm{LB}^{\mathrm{PALM}}(X_h)\le \min_{x\in X_h}\mu(x)$.
Hence, solving the PALM-Mean node problem \eqref{eq:PALM-lb} defines a valid lower bound for spatial branch-and-bound.
\end{proposition}

\begin{proof}
Fix a node $X_h$ and an arbitrary point $x\in X_h$. For each kernel function at important index $i\in\mathcal{I}(h)$, the piecewise linear envelopes are constructed such that
\[
\underline{k}_{i,h}^{P}(x)\le k_i(x)\le \overline{k}_{i,h}^{P}(x).
\]
Applying Proposition~\ref{prop:sign_aware_term_bounds} to these kernel bounds gives
\[
\alpha_i^+\underline{k}_{i,h}^{P}(x)-\alpha_i^-\overline{k}_{i,h}^{P}(x)\le \alpha_i k_i(x).
\]
For each kernel function at unimportant index $i\in\mathcal{U}(h)$, the analytical bounds in Appendix~\ref{app:analytic_bounds} satisfy $\underline{t}_{i,h}^{A}\le \alpha_i k_i(x)$. Summing the important-term inequalities and the unimportant-term inequalities over all indices gives
\[
\underline{\mu}_h^{\mathrm{PALM}}(x)\le \mu(x),
\qquad \forall x\in X_h.
\]
Since this inequality holds pointwise over the node subdomain, minimizing both sides over $x\in X_h$ yields
\[
\min_{x\in X_h}\underline{\mu}_h^{\mathrm{PALM}}(x)
\le
\min_{x\in X_h}\mu(x),
\]
which is exactly the claimed validity statement.
\end{proof}

\begin{proposition}
\label{prop:hybrid_consistency}
Let $\operatorname{diam}(X_h)$ denote the diameter of node subdomain $X_h$. Then, for each fixed index $i$, the gap between the exact term $\alpha_i k_i(x)$ and its PALM-Mean lower estimator on $X_h$ vanishes uniformly as $\operatorname{diam}(X_h)\to0$.
\end{proposition}

\begin{proof}
Consider first an analytically treated term. By construction, the lower and upper kernel bounds are obtained by evaluating a continuous monotone kernel profile at the minimum and maximum scaled distance from the node to the corresponding training point. As the node diameter tends to zero, these minimum and maximum distances converge to the same limit uniformly on the node. Continuity of the kernel profile therefore implies that the analytical kernel interval collapses uniformly, and the same is then true of the sign-aware analytical term interval.

Now consider a piecewise-treated term. Its lower and upper envelopes are defined in a scalar variable, either $d_i(x)$ for the RBF kernel or $r_i(x)$ for the Mat\'ern kernels. As the node diameter tends to zero, the corresponding scalar interval also shrinks to a point. On such shrinking intervals, the gap between a continuously differentiable function and its tangent/secant envelope vanishes uniformly. For the Mat\'ern-$3/2$ and Mat\'ern-$5/2$ kernels, if an interval crosses an inflection point, the construction inserts that point into the breakpoint set and applies the same argument separately on the concave and convex pieces. Hence the piecewise kernel envelopes converge uniformly to the exact kernel value. Applying the sign-aware transformation preserves this vanishing-gap property for each signed term contribution.
\end{proof}

\begin{proposition}
\label{thm:PALM_convergence}
Assume that each node lower-bounding problem is solved to global optimality up to the prescribed numerical tolerance, that node upper bounds are computed from feasible evaluations of the true posterior mean, and that the branching rule generates a sequence of nested boxes whose diameters converge to zero along every infinite branch. A spatial branch-and-bound algorithm satisfying these standard assumptions terminates with an $\varepsilon$-global minimizer of \eqref{eq:mean_problem}, up to the solver tolerances used in the node subproblems.
\end{proposition}

\begin{proof}
Proposition~\ref{thm:PALM_validity} implies that every node lower bound is valid for the exact problem, so pruning cannot discard any node that might still contain a point whose objective value is smaller than the incumbent by more than the tolerance. Feasible evaluations of the true posterior mean provide valid upper bounds throughout the algorithm. Proposition~\ref{prop:hybrid_consistency} shows that, along any infinite refinement sequence, the discrepancy between the node lower estimator and the exact posterior mean vanishes. Together with the assumption that node diameters shrink to zero, this gives asymptotically exact lower bounding on any branch that is not pruned. Standard spatial branch-and-bound arguments (see, e.g., \cite{ref:reduced_space_bb}) then imply that the global lower bound converges to the global optimum from below while the incumbent upper bound converges from above, and termination at the requested absolute or relative tolerance yields an $\varepsilon$-global minimizer.
\end{proof}

The three essential points for correctness are therefore: (i) valid node lower bounds, (ii) feasible node upper bounds based on the exact posterior mean, and (iii) a branching algorithm that drives node diameters to zero. A few practical enhancements improve efficiency but are not part of the core proof. We use best-bound node selection, longest-edge bisection, optional root pre-partitioning, and adaptive node MIQCP termination gaps. These details are summarized fully in Appendix~\ref{app:full_algorithm}, where we also provide full pseudocode for the complete PALM-Mean implementation.

\section{Numerical experiments}
\label{sec:numerical-experiments}

In this section, we compare PALM-Mean to existing methods on a variety of problems. The experiments are organized around specific case studies so that the effect of the proposed hybrid lower bound can be assessed in settings of increasing difficulty. Throughout, the objective function is always the exact posterior mean in \eqref{eq:mean_problem}. Additional solver settings and dataset details are collected in Appendix~\ref{app:exp_details}.

\subsection{Experimental setup and baselines}

We consider both synthetic and real-world GP posterior mean minimization problems. For the synthetic studies, training data are sampled from multimodal benchmark functions or from samples generated from GP priors. For the real-world studies, the GP is trained directly on experimental data from the literature. Unless otherwise noted, all GP surrogates use an ARD RBF kernel with hyperparameters fitted by maximum likelihood estimation after mean-centering the responses.

A problem instance is declared solved when the deterministic global optimality gap satisfies the stopping criteria, where we define the absolute gap tolerance $\epsilon_{\text{abs}}=0.1$ and relative gap tolerance $\epsilon_{\text{rel}}=0.01$. On the synthetic benchmarks, all exact global solvers are given a time limit of $600$ seconds per instance. On the real-world applications, the time limits are $600$ seconds for the N-benzylation problem and $3600$ seconds for the autonomous additive manufacturing problem.

Our primary baselines are three state-of-the-art deterministic global solvers: SCIP \cite{ref:scip}, MAiNGO \cite{ref:maingo}, and BARON \cite{ref:baron}. These methods are appropriate baselines because they target the exact nonlinear objective; however, they do not exploit the posterior mean structure in the way PALM-Mean does. We report the performance of PALM-Mean in both single-core and $16$-core modes. In the latter, the node MIQCPs are solved using Gurobi \cite{gurobi} with multi-threading enabled, together with the practical enhancements described in the Appendix~\ref{app:full_algorithm}. We model the PWL functions using the native \texttt{setPWLObj()} function in Gurobi. Node upper bounds for PALM-Mean are obtained by local optimization of the true posterior mean using L-BFGS-B \cite{zhu1997algorithm}. The code to reproduce the experiments can be found on GitHub: \url{https://github.com/PaulsonLab/PALM-Mean}.

\subsection{Synthetic benchmarks}

\subsubsection{EggHolder function}

We begin with the two-dimensional EggHolder function, which provides a compact but highly multimodal test case:
\[
f(x)=-(47+x_2)\sin\!\left(\sqrt{\left|x_2+\frac{x_1}{2}+47\right|}\right)
-x_1\sin\!\left(\sqrt{\left|x_1-(x_2+47)\right|}\right),
\]
over the domain $x\in[-512,512]^2$. For each training size $N\in\{100,500,1000,1500\}$, we generate training data by Latin hypercube sampling \cite{ref:lhs, mckay2000comparison}, fit a GP posterior mean, and then globally minimize that posterior mean using each solver. Each setting is repeated $10$ times with different sampled training sets.

Figure~\ref{fig:eggholder} reports average CPU time and solved fraction. This case study is useful because it isolates scaling in the number of training points while keeping the dimension fixed and visually interpretable. PALM-Mean is the strongest method overall. In the single-core setting it solves all tested instances within the time limit, and the $16$-core configuration further reduces runtime. SCIP is the most competitive baseline, solving nearly all instances up to $N=1000$, but the fraction of solved problems drops sharply at $N=1500$. The performance of MAiNGO degrades earlier, and BARON becomes ineffective once the training set is moderately large.
This behavior is consistent with the structure of the proposed method. As $N$ grows, the posterior mean contains more kernel terms, but only a subset of them remains influential on any given branch-and-bound node. PALM-Mean benefits from exploiting that fact directly. By contrast, general-purpose global solvers must relax the full posterior expression more uniformly, and that burden grows quickly with the size of the training set.

\begin{figure}[t]
\centering
\includegraphics[width=\textwidth]{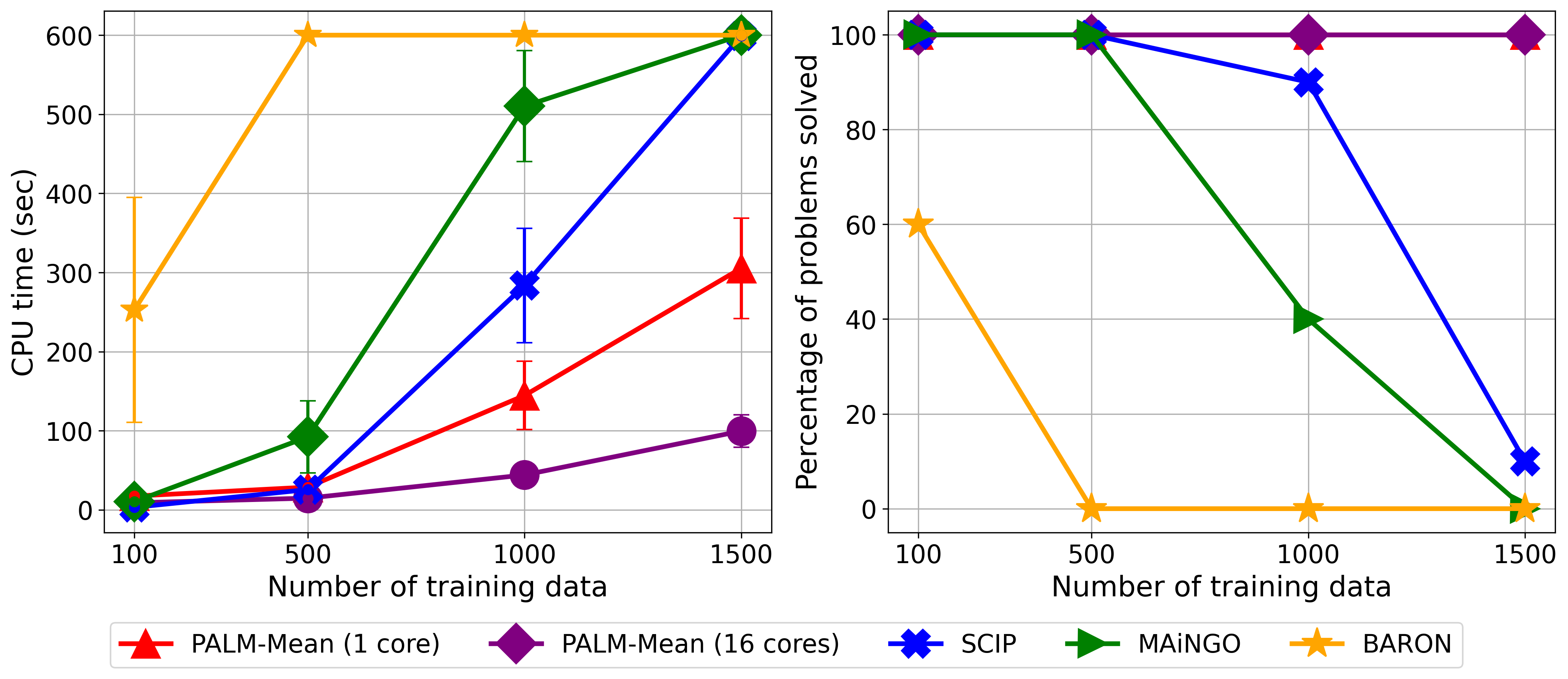}
\caption{Results for the EggHolder-based GP benchmark. Left: average CPU time over $10$ replicates, with error bars showing $\pm 1$ standard deviation. Right: percentage of instances solved within the $600$ second time limit.}
\label{fig:eggholder}
\end{figure}

\subsubsection{Schwefel function}

To study scaling in both dimension and training-set size, we next consider the Schwefel function, which is defined as follows
\[
f(x)=418.9829\,D-\sum_{j=1}^{D}x_j\sin\!\left(\sqrt{|x_j|}\right),
\]
over $[-500,500]^D$. As in the EggHolder study, the training data are generated by Latin hypercube sampling \cite{ref:lhs, mckay2000comparison}, a GP is fit to the sampled observations, and the resulting posterior mean is minimized globally. Each setting is repeated $40$ times with common random seeds across solvers. We consider two regimes: (i) a fixed-dimension study with $D=6$ and varying training size, and (ii) a fixed-training-size study with $N=400$ and varying input dimension.

Figure~\ref{fig:schwefel_cpu} reports CPU time, while Figure~\ref{fig:schwefel_solved} reports the corresponding solved fractions. The CPU-time plots show how quickly a solver converges for the instances that it is able to solve, whereas the solved-fraction plots reveal whether those average runtimes remain meaningful as the problems become harder.

We first consider the fixed-dimension study with $D=6$ and increasing training size. In Figure~\ref{fig:schwefel_cpu} (left), PALM-Mean has the lowest average CPU time across the full range of training sizes, with the $16$-core version giving the strongest performance overall. SCIP is the most competitive baseline, but its average runtime grows more rapidly as $N$ increases. MAiNGO and BARON become intractable much earlier and quickly reach the time limit. The solved-fraction results in Figure~\ref{fig:schwefel_solved} (left) highlights this trend. PALM-Mean with $16$ cores solves nearly all instances up to $N=500$ and still solves a large majority at $N=600$, while the solved fractions of the other deterministic solvers drop much more sharply. This is the regime where the hybrid structure of PALM-Mean is most directly beneficial. As before, as the number of training points grows, the posterior mean becomes a larger sum of kernel terms, but only a subset of those terms is locally important on any given node. PALM-Mean exploits that fact explicitly, whereas the other solvers relax the full posterior expression more uniformly.

We next consider the fixed-training-size study with $N=400$ and increasing dimension. Figure~\ref{fig:schwefel_cpu} (right) shows that PALM-Mean again has the best overall runtime profile, especially in the $16$-cores configuration. SCIP remains competitive through moderate dimension, but its runtime increases steadily and eventually becomes substantially larger. MAiNGO and BARON exhibit worse scaling with dimension, with BARON quickly becoming ineffective and MAiNGO losing competitiveness beyond the low-dimensional cases. The solved-fraction results in Figure~\ref{fig:schwefel_solved} (right) tell a similar story. PALM-Mean with $16$ cores remains the most reliable method throughout the entire dimensional range, while the solved fractions of SCIP, MAiNGO, and BARON decline more rapidly. The single-core PALM-Mean variant also remains competitive, although the benefit of parallelization becomes increasingly important in the higher-dimensional cases.

\begin{figure}[t]
\centering
\includegraphics[width=\textwidth]{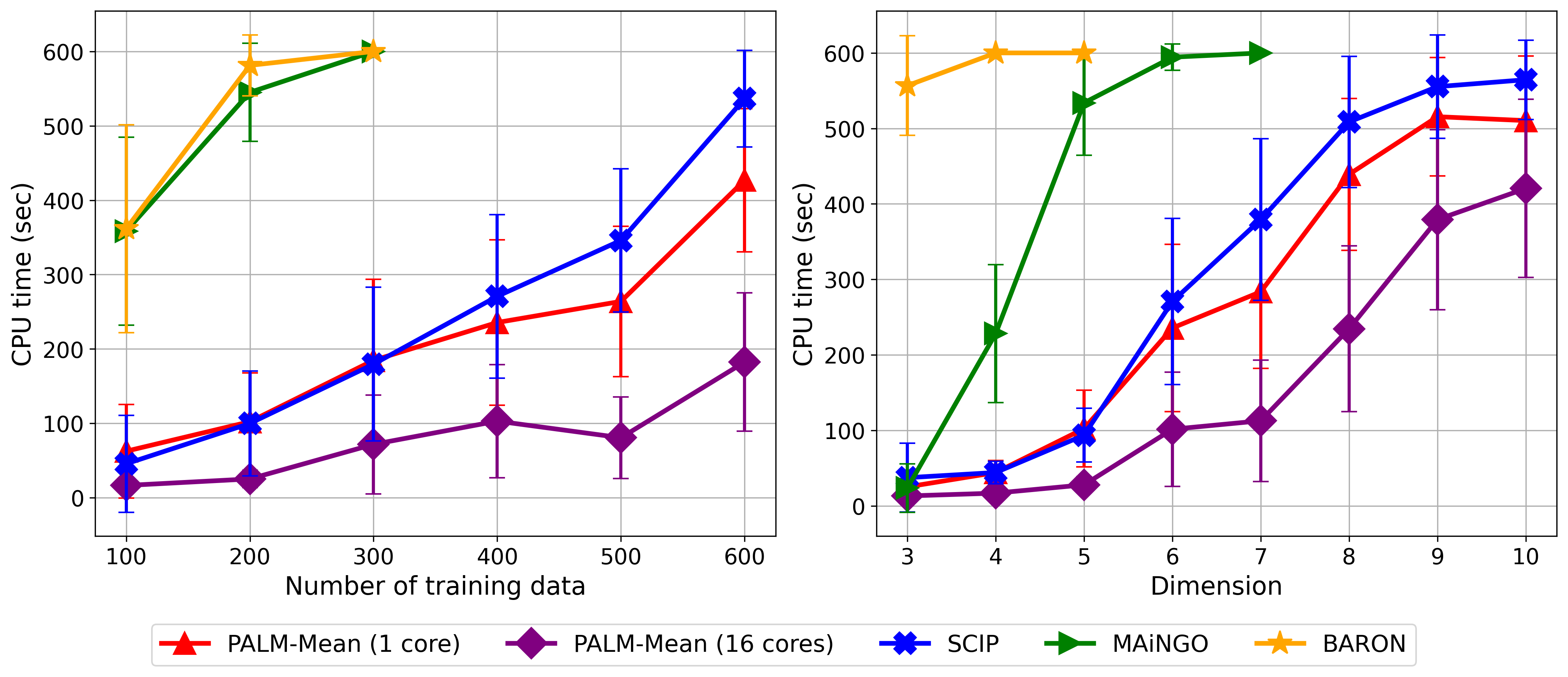}
\caption{CPU time on Schwefel-based GP benchmarks. Left: fixed dimension $D=6$ with varying number of training points. Right: fixed training size $N=400$ with varying dimension. Reported values are averages over $40$ replicates, with error bars corresponding to $\pm 1$ standard deviation.}
\label{fig:schwefel_cpu}
\end{figure}

\begin{figure}[t]
\centering
\includegraphics[width=\textwidth]{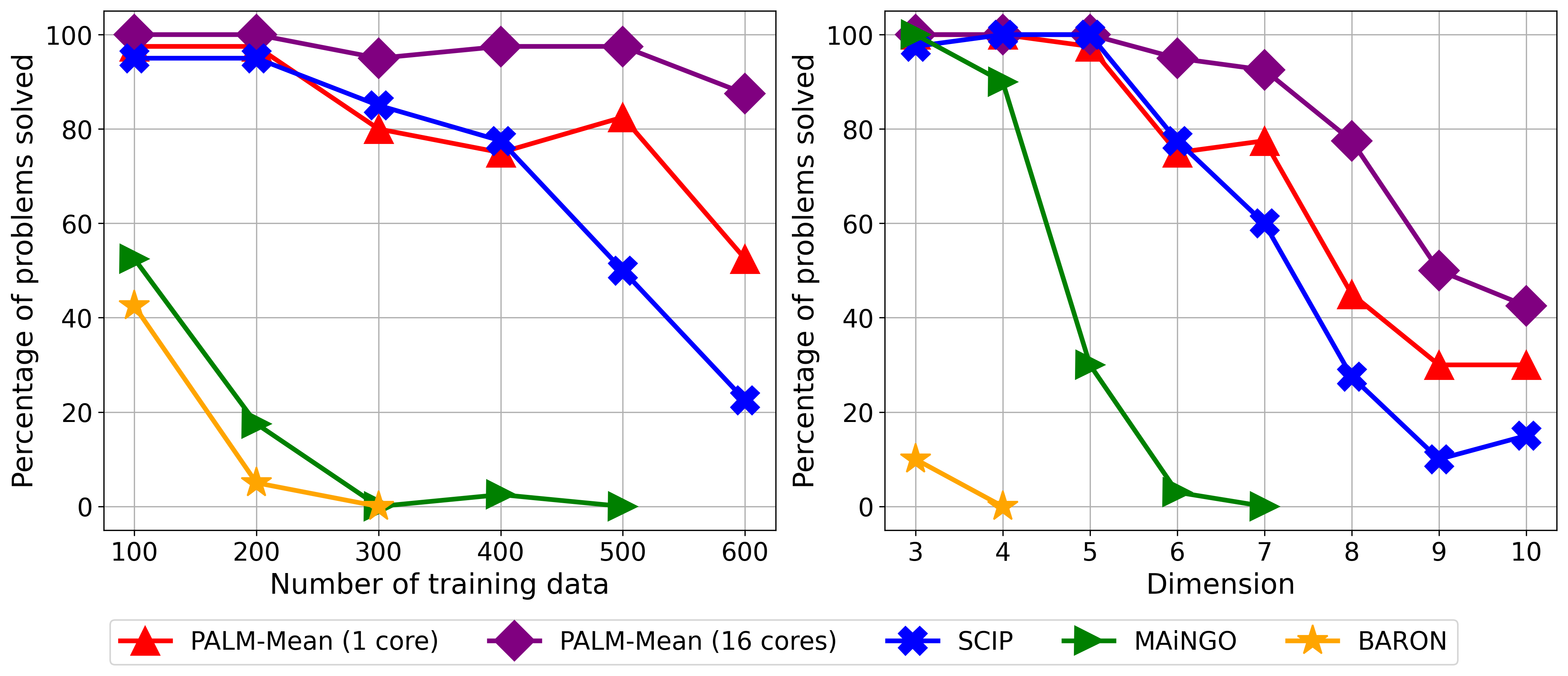}
\caption{Solved fraction on Schwefel-based GP benchmarks. Left: fixed dimension $D=6$ with varying number of training points. Right: fixed training size $N=400$ with varying dimension. Each point reports the percentage of instances solved within the $600$ second time limit over $40$ replicates.}
\label{fig:schwefel_solved}
\end{figure}

\subsubsection{Out-of-model GP-prior study}

The previous two studies are tied to specific (fixed) benchmark functions. To evaluate robustness across a broader family of posterior mean shapes, we also consider an out-of-model comparison based on GP prior realizations, which is common in BO benchmarking literature \cite{ref:bo_prior_benchmark_a,ref:bo_prior_benchmark_b}. Specifically, we generate $100$ objective functions by sampling from an oracle GP prior. For each sampled function, the dimension is drawn uniformly from $\{6,\ldots,10\}$, the ARD lengthscales are sampled from $\Gamma(2,0.1)$, and the number of training points is drawn uniformly from $\{200,\ldots,600\}$. A second GP is then fitted to the sampled observations (which are treated as black-box data), and we define the optimization task to minimize its posterior mean function.

Because the earlier synthetic studies show MAiNGO and BARON are uncompetitive in this data regime, we restrict this experiment to PALM-Mean and SCIP. Figure~\ref{fig:within_model} reports the cumulative fraction of solved instances as a function of cumulative CPU time. PALM-Mean with $16$ cores clearly outperforms the other approaches, both in final solved fraction and in the rate at which problem instances are solved. The single-core version is also competitive with SCIP, and in cumulative time it reaches a similar solved fraction earlier.

\begin{figure}[t]
\centering
\includegraphics[width=0.7\textwidth]{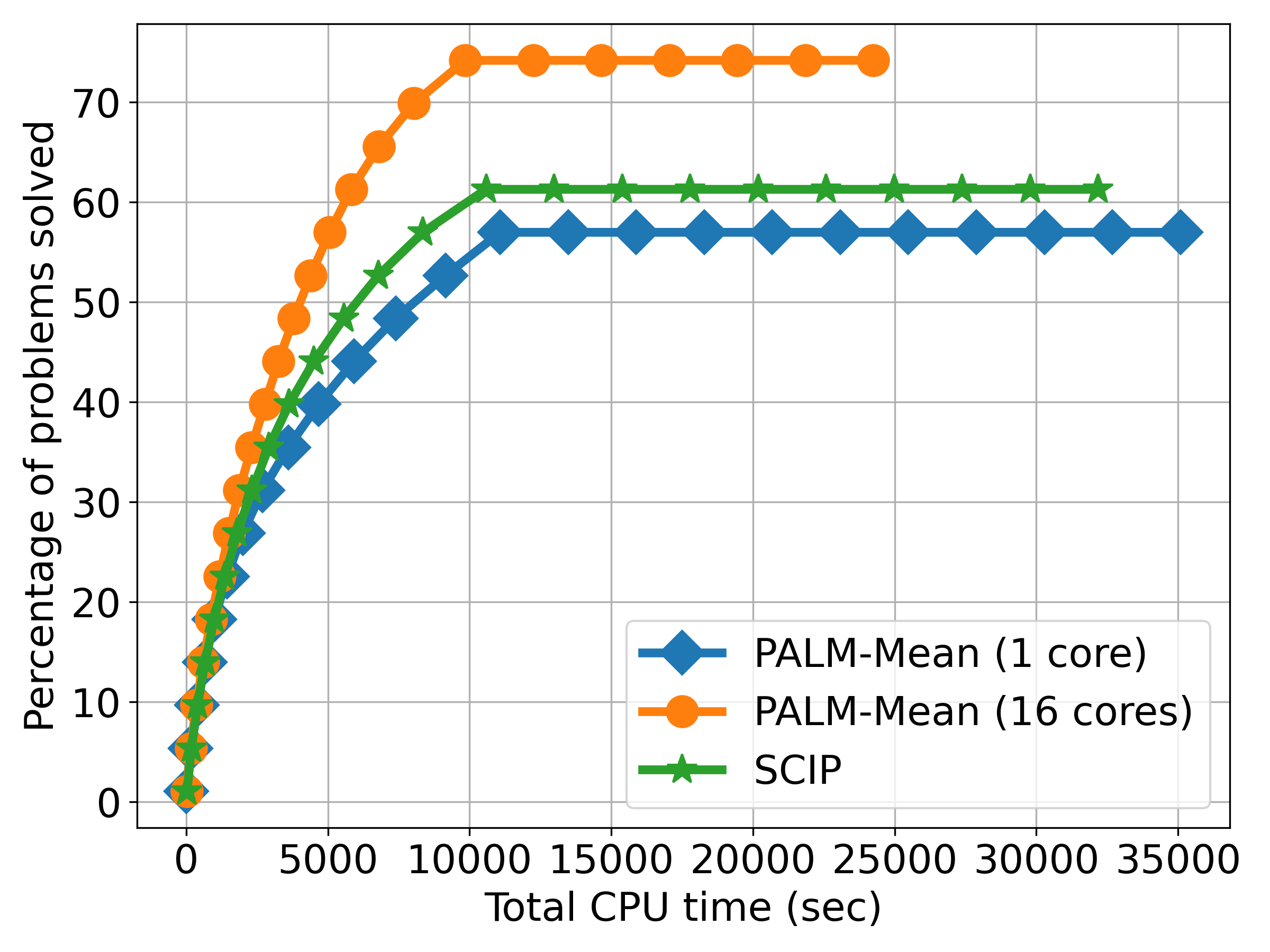}
\caption{Within-model comparison based on $100$ realizations of GP prior functions. The vertical axis reports the cumulative fraction of solved instances, and the horizontal axis reports cumulative CPU time.}
\label{fig:within_model}
\end{figure}

\subsection{Real-world benchmark problems}

We next consider two application-driven GP posterior mean optimization problems built from experimental data. These case studies are important because they test the method on trained surrogates that arise from realistic modeling workflows rather than only samples generated from synthetic functions.

\subsubsection{N-benzylation reaction}

Our first real-world benchmark is adapted from the N-benzylation reaction study in \cite{ref:gp_embedded_exact}. The original problem uses two GP models to describe the space-time yields (STY) of a desired secondary ($2^\circ$) amine product and an undesired tertiary ($3^\circ$) amine byproduct. To keep the present paper focused on posterior mean minimization, we convert the original constrained formulation into a penalized unconstrained problem using posterior means as surrogates:
\[
\min_{x\in\mathcal{X}}
-\left(
2^{\circ}\mathrm{STY}
-
\rho \max\!\big(3^{\circ}\mathrm{STY}-5,0\big)
\right),
\qquad \rho=100.
\]
There are total of 78 experimental data points in the training set. The four decision variables are the primary-amine feed flowrate, the benzyl-bromide-to-amine ratio, the solvent-to-amine ratio, and the operating temperature, with bounds
\[
x^L=[0.2,1.0,0.5,110],
\qquad
x^U=[0.4,5.0,1.0,150].
\]
Table~\ref{tab:benzylation_results} shows that PALM-Mean is substantially faster than the competing exact solvers. The single-core and $16$-cores variants both solve the problem to the target absolute gap, with the multi-threaded version requiring the least CPU time. SCIP also solves the problem within the alloted time, but it requires orders of magnitude more branch-and-bound nodes. MAiNGO and BARON both reach the time limit with weaker lower bounds.

\begin{table}[t]
    \centering
    \caption{Results for the N-benzylation reaction benchmark with a time limit of $600$ seconds.}
    \label{tab:benzylation_results}
    \begin{tabular}{lccccc}
        \toprule
        Method & Nodes & CPU time (s) & UB & LB & Abs.\ gap \\
        \midrule
        PALM-Mean (16 cores) & \textbf{110} & \textbf{15.4} & $-1.43$ & $-1.53$ & $0.10$ \\
        PALM-Mean (1 core)   & 110          & 33.5          & $-1.43$ & $-1.53$ & $0.10$ \\
        SCIP                 & 66,222       & 106.9         & $-1.43$ & $-1.53$ & $0.10$ \\
        MAiNGO               & 609,292      & 600           & $-1.43$ & $-4.95$ & $3.52$ \\
        BARON                & 45,861       & 600           & $-1.43$ & $-23.24$ & $21.81$ \\
        \bottomrule
    \end{tabular}
\end{table}

\subsubsection{Autonomous additive manufacturing}

Our second real-world benchmark is focused on optimization of printing conditions for autonomous additive manufacturing. We train a GP on the full AutoAM dataset with 100 experimental data \cite{ref:autoam_dataset,ref:autoam_bo}, and minimize the negative ratio between effective and desired specimen area:
\[
\min_{x\in\mathcal{X}} -\frac{A_{\mathrm{effective}}}{A_{\mathrm{desired}}}.
\]
The four decision variables are prime delay, print speed, $x$-offset correction, and $y$-offset correction, with bounds
\[
x^L=[0,0,-1,-1],
\qquad
x^U=[5,10,1,1].
\]

This is a harder instance than the N-benzylation problem and more clearly exposes the value of the implementation enhancements summarized in Appendix~\ref{app:full_algorithm}. As shown in Table~\ref{tab:autoam_results}, only PALM-Mean with $16$ cores and SCIP are able to solve the problem to the target absolute gap within the one-hour time limit. PALM-Mean is faster and does so with a dramatically smaller branch-and-bound tree. The single-core version does not fully solve the problem within the time limit, but its final gap is still small. MAiNGO and BARON again terminate with weak lower bounds.

\begin{table}[t]
    \centering
    \caption{Results for the autonomous additive manufacturing benchmark with a time limit of $3600$ seconds.}
    \label{tab:autoam_results}
    \begin{tabular}{lccccc}
        \toprule
        Method & Nodes & CPU time (s) & UB & LB & Abs.\ gap \\
        \midrule
        PALM-Mean (16 cores) & \textbf{4,423} & \textbf{1,074.1} & $-1.60$ & $-1.70$ & $0.10$ \\
        PALM-Mean (1 core)   & 3,163          & 3,600            & $-1.60$ & $-1.79$ & $0.19$ \\
        SCIP                 & 523,427        & 1,734.8          & $-1.60$ & $-1.70$ & $0.10$ \\
        MAiNGO               & 945,300        & 3,600            & $-1.60$ & $-8.24$ & $6.64$ \\
        BARON                & 69,958         & 3,600            & $-1.60$ & $-24.75$ & $23.15$ \\
        \bottomrule
    \end{tabular}
\end{table}

\subsection{Discussion}

Several conclusions are consistent across the considered case studies. First, the main advantage of PALM-Mean is not that it uses piecewise-linear envelopes, but rather that it uses them selectively, only for kernel terms that are locally important on the current node. The benefit of that design is most apparent when the number of training points is large, since that is when relaxing the full posterior sum becomes difficult for existing (more general-purpose) solvers.

Second, the performance gains are closely tied to search-tree size. On the real-world problems, PALM-Mean reduces the number of explored nodes by orders of magnitude relative to SCIP and MAiNGO. That is strong evidence that the hybrid lower bound is creating a smaller tree by enabling earlier pruning.

Third, the results also clarify the respective roles of local, exact global, and approximation-based methods. Local optimization remains useful as a fast incumbent generator, and PALM-Mean relies on it for upper bounding. Generic exact solvers can certify global optimality, but their relaxations become weak as the posterior mean grows in size and complexity. Approximation-based methods can improve tractability, but they do so by solving a modified problem. PALM-Mean occupies a useful middle ground in that it remains an exact deterministic method for posterior mean minimization, but it exploits the structure of the GP objective strongly enough to make substantially larger instances tractable.

However, PALM-Mean does not remove the intrinsic hardness of global optimization, and the node MIQCPs can still become expensive when many training points remain important over large parts of the domain or when the posterior mean is highly oscillatory. This is visible in the hardest synthetic settings and in the single-core AutoAM results. Even so, the results consistently show that the hybrid piecewise-analytic lower bound materially improves scalability relative to existing alternatives on exact posterior mean minimization problems.

\section{Conclusion}
\label{sec:conclusions}

This paper studies the deterministic global minimization of trained Gaussian process (GP) posterior mean functions over a hyperrectangular domain. Although the posterior mean admits a compact closed-form expression, its global optimization remains challenging because it is a non-convex weighted sum of kernel terms centered at all training points. Existing exact deterministic approaches target the original objective function, but can become difficult to scale as the posterior sum grows, while approximation-based approaches improve tractability by optimizing a modified objective.
To address this gap, we introduce PALM-Mean, a piecewise-analytic spatial branch-and-bound method tailored specifically to posterior mean minimization. The central idea is to combine tighter piecewise-linear treatment for nodewise important kernel terms with inexpensive analytical bounds for the remaining terms, yielding a valid hybrid lower bound for the exact posterior mean. This preserves deterministic correctness while allocating computational effort where it is most useful for pruning.

Our computational results show that this structure can make a substantial practical difference. Across synthetic benchmarks and real application problems, PALM-Mean consistently improves scalability relative to representative general-purpose deterministic global solvers, especially as the number of training points increased. The gains are reflected not only in actual runtime, but also in much smaller branch-and-bound trees, which is consistent with the stronger node lower bounds produced by the hybrid construction.
From a practical standpoint, these results show that deterministic global optimization of GP posterior mean functions does not need to rely solely on general-purpose nonlinear global solvers. When the objective has the specific structure of a trained GP posterior mean, that structure can be leveraged directly to obtain stronger relaxations and more effective global search. This is particularly relevant in low-data scientific and engineering settings, where GPs are a common surrogate model and certified global decisions can be practically important. 

Lastly, we note that the present paper is intentionally focused on the posterior mean problem, which provides the clearest setting in which to develop and analyze the piecewise-analytic perspective. At the same time, the same viewpoint suggests future extensions to broader GP posterior objectives. We view the mean-only case developed here as the foundation for those next steps. Conceptually, the same idea proposed in this work can be extended to optimizing the GP posterior variance or the lower confidence bound. However, the double summation over kernel products for the posterior variance could provide significant computational challenge to solve the lower bound. The development of extension methods tailored to problems involved with posterior variance is an interesting future research direction.

\bmhead{Acknowledgments} WTT, AK, and JAP were partially supported by the National Science Foundation (NSF) CAREER Award 2237616. CT was supported by a BASF/Royal Academy of Engineering Senior Research Fellowship. 

\section*{Declarations}

\subsection*{Data availability}
No new experimental data were generated in this study. The synthetic benchmark instances used in the numerical experiments can be generated using the scripts provided in the code repository. The literature-based data used for the real-world benchmark problems are available from the sources cited in the article. 

\subsection*{Code availability}
The code used to reproduce the numerical experiments and figures in this study is available from the project GitHub repository: \url{https://github.com/PaulsonLab/PALM-Mean}.

\appendix

\section{Additional technical details for PALM-Mean}

This appendix collects formulas and implementation details that are omitted from the main text for brevity. The goal is to make the lower-bounding construction and the resulting spatial branch-and-bound implementation sufficiently explicit that a reader can reproduce the method.

\subsection{Closed-form kernel bounds over a hyperrectangle}
\label{app:analytic_bounds}

Let $X_h=[x_h^L,x_h^U]\subseteq\mathcal{X}$ be a node of the spatial branch-and-bound tree, and fix a training point $x_i\in\mathbb{R}^D$. Recall that
\[
r_i(x)=\|L^{-1}(x-x_i)\|_2,
\qquad
d_i(x)=r_i(x)^2.
\]
Because $L=\mathrm{diag}(\ell_1,\ldots,\ell_D)$ is diagonal and $X_h$ is axis-aligned, the minimum and maximum of the scaled distance over $X_h$ can be computed coordinatewise.

It is useful to define coordinatewise nearest-distance and farthest-distance values
\[
\Delta^{\mathrm{near}}_{i,h,j}
:=
\begin{cases}
0, & x_{h,j}^L \le x_{i,j} \le x_{h,j}^U,\\
x_{h,j}^L - x_{i,j}, & x_{i,j} < x_{h,j}^L,\\
x_{i,j} - x_{h,j}^U, & x_{i,j} > x_{h,j}^U,
\end{cases}
\]
and
\[
\Delta^{\mathrm{far}}_{i,h,j}
:=
\max\!\left\{|x_{h,j}^L-x_{i,j}|,\;|x_{h,j}^U-x_{i,j}|\right\},
\qquad j=1,\ldots,D.
\]
Equivalently, the nearest point is the coordinatewise projection of $x_i$ onto $X_h$ and the farthest point is a farthest vertex of $X_h$.
The resulting min/max scaled distances are
\begin{equation}
r_{i,h}^L
=
\left(
\sum_{j=1}^D \frac{\big(\Delta^{\mathrm{near}}_{i,h,j}\big)^2}{\ell_j^2}
\right)^{1/2},
\qquad
r_{i,h}^U
=
\left(
\sum_{j=1}^D \frac{\big(\Delta^{\mathrm{far}}_{i,h,j}\big)^2}{\ell_j^2}
\right)^{1/2}.
\label{eq:app_r_bounds}
\end{equation}
For the RBF kernel, the corresponding bounds on squared scaled distance are
\[
d_{i,h}^L=(r_{i,h}^L)^2,
\qquad
d_{i,h}^U=(r_{i,h}^U)^2.
\]

These distances immediately yield analytical box bounds for the kernel term $k_i(x)=k(x,x_i)$. For the RBF kernel,
\[
\underline{k}_{i,h}^{A}
=
\sigma_f^2 \exp\!\left(-\frac{d_{i,h}^U}{2}\right),
\qquad
\overline{k}_{i,h}^{A}
=
\sigma_f^2 \exp\!\left(-\frac{d_{i,h}^L}{2}\right),
\]
and for the Mat\'ern-$\nu$ kernels with $\nu\in\{1/2,3/2,5/2\}$,
\[
\underline{k}_{i,h}^{A}
=
\kappa_\nu(r_{i,h}^U),
\qquad
\overline{k}_{i,h}^{A}
=
\kappa_\nu(r_{i,h}^L).
\]
Since each kernel profile is monotonically decreasing in its scalar distance argument, these constants satisfy
\[
\underline{k}_{i,h}^{A}\le k_i(x)\le \overline{k}_{i,h}^{A},
\qquad \forall x\in X_h.
\]
Then, applying Proposition~\ref{prop:sign_aware_term_bounds} directly yields the analytical term bounds
\[
\underline{t}_{i,h}^{A}
=
\alpha_i^+ \underline{k}_{i,h}^{A}
-
\alpha_i^- \overline{k}_{i,h}^{A},
\qquad
\overline{t}_{i,h}^{A}
=
\alpha_i^+ \overline{k}_{i,h}^{A}
-
\alpha_i^- \underline{k}_{i,h}^{A}.
\]

\subsection{Detailed tangent/secant formulas for RBF and Mat\'ern kernel envelopes}
\label{app:tangent_secant}

For each important term $i\in\mathcal{I}(h)$, PALM-Mean constructs a piecewise-linear envelope with respect to a scalar distance variable
\[
\zeta_i
:=
\begin{cases}
d_i(x), & \text{RBF kernel},\\
r_i(x), & \text{Mat\'ern kernels}.
\end{cases}
\]
Let the node-specific scalar range be $[\zeta_{i,h}^L,\zeta_{i,h}^U]$, and choose breakpoints
\[
\zeta_{i,h}^L=\eta_{i,0}<\eta_{i,1}<\cdots<\eta_{i,S_i}=\zeta_{i,h}^U.
\]
On each segment $[\eta_{i,s-1},\eta_{i,s}]$, the secant line for a generic scalar kernel profile $\kappa$ is
\begin{equation}
u_{i,s}(\zeta)
=
m^{\mathrm{sec}}_{i,s}\zeta+b^{\mathrm{sec}}_{i,s},
\qquad
m^{\mathrm{sec}}_{i,s}
=
\frac{\kappa(\eta_{i,s})-\kappa(\eta_{i,s-1})}{\eta_{i,s}-\eta_{i,s-1}},
\label{eq:app_secant}
\end{equation}
with
\[
b^{\mathrm{sec}}_{i,s}
=
\kappa(\eta_{i,s-1})-m^{\mathrm{sec}}_{i,s}\eta_{i,s-1}.
\]


Within nearby breakpoints $[\eta_{i,s-1},\eta_{i,s}]$, the tangent envelope for a generic scalar kernel profile $\kappa$ is
\begin{equation}
\ell_{i,s}(\zeta)= \max(m^{\mathrm{tan}}_{i,s-1}\zeta+b^{\mathrm{tan}}_{i,s-1}, m^{\mathrm{tan}}_{i,s}\zeta+b^{\mathrm{tan}}_{i,s}),~
m^{\mathrm{tan}}_{i,s-1}=\kappa'(\eta_{i,s-1}), ~ m^{\mathrm{tan}}_{i,s}=\kappa'(\eta_{i,s})
\label{eq:app_tangent}
\end{equation}
with
\begin{align*}
    b^{\mathrm{tan}}_{i,s-1} = \kappa(\eta_{i,s-1}) - \kappa'(\eta_{i,s-1})\eta_{i,s-1}, \qquad b^{\mathrm{tan}}_{i,s} = \kappa(\eta_{i,s}) - \kappa'(\eta_{i,s})\eta_{i,s}
\end{align*}
Notably, a tangent line constructed at any point within the interval (e.g., at the midpoint) also provides a valid lower-bounding envelope. A systematic comparison of the tightness of such single-point tangent approximations versus our multi-tangent envelope, as well as their respective computational costs, represents an interesting direction for future work.

\paragraph{RBF kernel.}
For the RBF kernel,
\[
\kappa_{\mathrm{RBF}}(d)=\sigma_f^2 e^{-d/2},
\qquad
\kappa_{\mathrm{RBF}}'(d)=-\frac{\sigma_f^2}{2}e^{-d/2},
\qquad
\kappa_{\mathrm{RBF}}''(d)=\frac{\sigma_f^2}{4}e^{-d/2}.
\]
Since $\kappa_{\mathrm{RBF}}$ is convex and decreasing on $[0,\infty)$, the tangent line is a valid lower bound and the secant line is a valid upper bound on every segment:
\[
\underline{\kappa}^{P}_{i,s}(d)=\ell_{i,s}(d),
\qquad
\overline{\kappa}^{P}_{i,s}(d)=u_{i,s}(d).
\]

\paragraph{Mat\'ern-$1/2$.}
For the Mat\'ern-$1/2$ kernel,
\[
\kappa_{1/2}(r)=\sigma_f^2 e^{-r},
\qquad
\kappa_{1/2}'(r)=-\sigma_f^2 e^{-r},
\qquad
\kappa_{1/2}''(r)=\sigma_f^2 e^{-r}.
\]
Thus, $\kappa_{1/2}$ is also convex and decreasing on $[0,\infty)$, so the same rule applies.

\paragraph{Mat\'ern-$3/2$.}
For the Mat\'ern-$3/2$ kernel,
\[
\kappa_{3/2}(r)=\sigma_f^2(1+\sqrt{3}\,r)e^{-\sqrt{3}r},
\]
\[
\kappa_{3/2}'(r)=-3\sigma_f^2 r e^{-\sqrt{3}r},
\qquad
\kappa_{3/2}''(r)=3\sigma_f^2(\sqrt{3}r-1)e^{-\sqrt{3}r}.
\]
The inflection point is
\[
r_{3/2}^{\mathrm{infl}}=\frac{1}{\sqrt{3}}.
\]
If $r_{i,h}^L<r_{3/2}^{\mathrm{infl}}<r_{i,h}^U$, this inflection point is inserted into the breakpoint set. On segments lying in the concave region $[0,r_{3/2}^{\mathrm{infl}}]$, the secant is a lower bound and the tangent is an upper bound. On segments in the convex region $(r_{3/2}^{\mathrm{infl}},\infty)$, these roles are reversed.

\paragraph{Mat\'ern-$5/2$.}
For the Mat\'ern-$5/2$ kernel,
\[
\kappa_{5/2}(r)
=
\sigma_f^2\left(1+\sqrt{5}r+\frac{5}{3}r^2\right)e^{-\sqrt{5}r},
\]
\[
\kappa_{5/2}'(r)
=
-\frac{5\sigma_f^2}{3}r(1+\sqrt{5}r)e^{-\sqrt{5}r},
\]
\[
\kappa_{5/2}''(r)
=
\frac{5\sigma_f^2}{3}\left(5r^2-\sqrt{5}r-1\right)e^{-\sqrt{5}r}.
\]
The positive inflection point is
\[
r_{5/2}^{\mathrm{infl}}
=
\frac{1+\sqrt{5}}{2\sqrt{5}}.
\]
As for Mat\'ern-$3/2$, this inflection point is inserted into the breakpoint set whenever it lies inside the node interval.

\paragraph{Segmentwise sign-aware bounds.}
Once the lower and upper kernel envelopes have been defined on a segment, the associated term bounds are given by
\[
\underline{t}^{P}_{i,s}(\zeta_i)
=
\alpha_i^+ \underline{\kappa}^{P}_{i,s}(\zeta_i)
-
\alpha_i^- \overline{\kappa}^{P}_{i,s}(\zeta_i),
\]
\[
\overline{t}^{P}_{i,s}(\zeta_i)
=
\alpha_i^+ \overline{\kappa}^{P}_{i,s}(\zeta_i)
-
\alpha_i^- \underline{\kappa}^{P}_{i,s}(\zeta_i).
\]

\subsection{Mixed-integer encoding of the piecewise-linear bounds}
\label{app:mip_encoding}

We now give one explicit mixed-integer encoding of the piecewise-linear bounds used for the important terms. This appendix uses a segment-activation formulation. Solver-native special ordered sets of type 2 (SOS2) constraints\footnote{\url{https://lpsolve.sourceforge.net/5.5/SOS.htm}} can also be used in practice, but the form below makes the lower-bounding model explicit.

Fix a node $X_h$ and an important index $i\in\mathcal{I}(h)$. Let
\[
[\eta_{i,0},\eta_{i,1}],\ldots,[\eta_{i,S_i-1},\eta_{i,S_i}]
\]
denote the segments of the scalar interval for $\zeta_i$. Introduce binary variables
\[
z_{i,s}\in\{0,1\},
\qquad s=1,\ldots,S_i,
\]
with the following single-segment activation constraint
\begin{equation}
\sum_{s=1}^{S_i} z_{i,s}=1.
\label{eq:app_one_segment}
\end{equation}
Let $p_i$ and $q_i$ denote the lower- and upper-envelope values used in the node objective. For each segment $s$, write the active lower and upper affine functions as
\[
\underline{\kappa}^{P}_{i,s}(\zeta)=m^L_{i,s}\zeta+b^L_{i,s},
\qquad
\overline{\kappa}^{P}_{i,s}(\zeta)=m^U_{i,s}\zeta+b^U_{i,s}.
\]
Note that the formulation presented here uses a single tangent line to construct the lower affine function for each segment. As discussed in Appendix~\ref{app:tangent_secant}, however, we adopt a different formulation in this work, where the lower affine function is defined as the maximum of two tangent lines evaluated at the nearby segment points. For simplicity of exposition, we focus on the single-tangent formulation in this section when describing the lower envelope. A valid disjunctive formulation is then
\begin{equation}
\eta_{i,s-1}-M_{i,s}^{\zeta}(1-z_{i,s})
\le \zeta_i \le
\eta_{i,s}+M_{i,s}^{\zeta}(1-z_{i,s}),
\qquad s=1,\ldots,S_i,
\label{eq:app_seg_interval}
\end{equation}
\begin{equation}
-M_{i,s}^{p}(1-z_{i,s})
\le
p_i-\left(m^L_{i,s}\zeta_i+b^L_{i,s}\right)
\le
M_{i,s}^{p}(1-z_{i,s}),
\qquad s=1,\ldots,S_i,
\label{eq:app_seg_p}
\end{equation}
\begin{equation}
-M_{i,s}^{q}(1-z_{i,s})
\le
q_i-\left(m^U_{i,s}\zeta_i+b^U_{i,s}\right)
\le
M_{i,s}^{q}(1-z_{i,s}),
\qquad s=1,\ldots,S_i.
\label{eq:app_seg_q}
\end{equation}
When solver-native SOS2 support is available, we prefer that implementation because it avoids manual tuning of big-$M$ values.

The exact link between $\zeta_i$ and the decision vector $x$ remains
\begin{equation}
\zeta_i=(x-x_i)^\top L^{-2}(x-x_i),
\qquad i\in\mathcal{I}(h),
\label{eq:app_rbf_link}
\end{equation}
for the RBF kernel, and
\begin{equation}
\zeta_i^2=(x-x_i)^\top L^{-2}(x-x_i),
\qquad
\zeta_i\ge0,
\qquad i\in\mathcal{I}(h),
\label{eq:app_matern_link}
\end{equation}
for the Mat\'ern kernels.

Combining the important-term piecewise variables with the unimportant-term analytical constants yields the node lower-bounding problem
\begin{equation}
\min_{x,\zeta,p,q,z}
\quad
\sum_{i\in\mathcal{I}(h)}
\left(\alpha_i^+ p_i-\alpha_i^- q_i\right)
+
\sum_{i\in\mathcal{U}(h)}
\underline{t}_{i,h}^{A}
\label{eq:app_full_node_obj}
\end{equation}
subject to $x\in X_h$, the piecewise constraints \eqref{eq:app_one_segment}-\eqref{eq:app_seg_q}, and the exact distance link \eqref{eq:app_rbf_link} or \eqref{eq:app_matern_link}. This is a non-convex MIQCP. The integer structure enters through the segment indicators $z_{i,s}$, and the non-convexity enters through the exact quadratic equalities linking $\zeta_i$ to $x$.

\subsection{Full PALM-Mean pseudocode and implementation details}
\label{app:full_algorithm}

Figure~\ref{fig:overview_appendix} provides an overview of the PALM-Mean workflow. The complete implementation is summarized in Algorithm~\ref{alg:palm_mean_full}. The logic is standard spatial branch-and-bound, but we include the optional preprocessing and adaptive node-tolerance features explicitly.

\begin{figure}[t]
\centering
\includegraphics[width=0.8\textwidth]{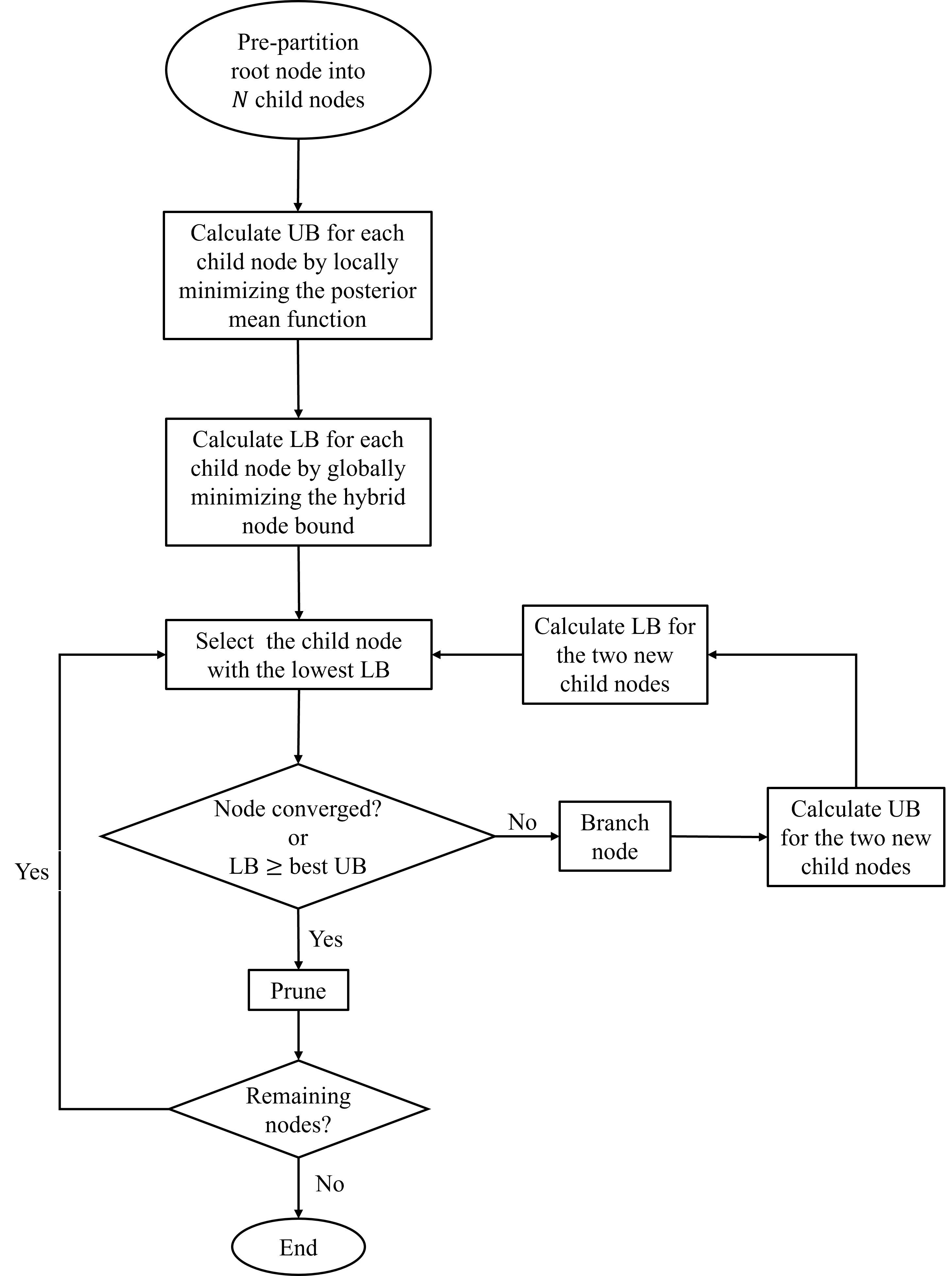}
\caption{Overview of the PALM-Mean workflow.}
\label{fig:overview_appendix}
\end{figure}

\begin{algorithm}[t]
\caption{PALM-Mean}
\label{alg:palm_mean_full}
\begin{algorithmic}[1]
\State \textbf{Input:} root box $\mathcal{X}$, trained GP quantities, importance threshold $\rho$, tolerances $\epsilon_{\mathrm{abs}}$, $\epsilon_{\mathrm{rel}}$, and node-MIQCP bounds $\epsilon_{\min}^{\mathrm{MIQCP}}$, $\epsilon_{\max}^{\mathrm{MIQCP}}$
\State \textbf{Initialize active list:} use $\{\mathcal{X}\}$ or an optional collection of root subboxes
\State \textbf{Initialize incumbent:} $\mathrm{UB}^\star\leftarrow +\infty$
\ForAll{$X_h$ in the active list}
    \State Compute $\mathrm{UB}(X_h)$ by local minimization of the true posterior mean on $X_h$
    \State Update incumbent if $\mathrm{UB}(X_h)<\mathrm{UB}^\star$
    \State Compute nodewise scores $\beta_{i,h}$ and form $\mathcal{I}(h)$ and $\mathcal{U}(h)$
    \State Build the hybrid node MIQCP and solve it globally to obtain $\mathrm{LB}^{\mathrm{PALM}}(X_h)$
\EndFor
\State Prune nodes whose lower bound exceeds the incumbent within tolerance
\While{global gap exceeds $\epsilon_{\mathrm{abs}}$ and $\epsilon_{\mathrm{rel}}$}
    \State Select the active node with the smallest lower bound
    \State Bisect the node along its longest edge
    \For{each child node}
        \State Compute feasible upper bound by local minimization of true posterior mean
        \State Update incumbent if improved
        \State Recompute $\mathcal{I}(h)$ and $\mathcal{U}(h)$
        \State Build and globally solve the hybrid node MIQCP
        \State Prune if the node cannot improve the incumbent
    \EndFor
\EndWhile
\State \textbf{Return:} incumbent solution and certified objective value
\end{algorithmic}
\end{algorithm}

A few implementation remarks are worth recording explicitly.
\begin{enumerate}
    \item \textbf{Node selection.} We use a best-bound strategy, i.e., the active node with the smallest lower bound is processed next (see Line 12).
    \item \textbf{Branching rule.} We bisect the current node along its longest side.
    \item \textbf{Upper bounding.} We use L-BFGS-B to minimize the true posterior mean over each node.
    \item \textbf{Optional preprocessing.} When root pre-partitioning is enabled, each dimension is partitioned proportionally to its side length so that the initial subboxes have comparable diagonal lengths.
    \item \textbf{Adaptive node tolerances.} If a node-local upper bound is already worse than the incumbent, the corresponding MIQCP does not need to be solved to the tightest possible tolerance. In our implementation, the absolute node MIQCP tolerance is clipped according to
    \[
    \epsilon_h^{\mathrm{MIQCP}}
    =
    \max\!\left\{
    \epsilon_{\min}^{\mathrm{MIQCP}},
    \min\!\left[
    \epsilon_{\max}^{\mathrm{MIQCP}},
    \mathrm{UB}(X_h)-\mathrm{UB}^\star
    \right]
    \right\}.
    \]
\end{enumerate}

\subsection{Additional experimental details and solver settings}
\label{app:exp_details}

This subsection records the main experimental settings used in Section~\ref{sec:numerical-experiments}.

\paragraph{Common GP settings.}
Unless otherwise noted, each GP surrogate uses a zero prior mean after response centering and an ARD RBF kernel. Kernel hyperparameters are fitted by maximum likelihood estimation with \texttt{fit\_gpytorch\_mll} routine in BoTorch \cite{balandat2020botorch}. The posterior mean coefficients are then computed as
\[
\alpha=(K_{XX}+\sigma_\epsilon^2 I)^{-1}y,
\]
and the resulting posterior mean is minimized over the prescribed box domain.

\paragraph{Common solver settings.}
All synthetic benchmarks use a time limit of $600$ seconds per instance. The real-world N-benzylation benchmark uses the same $600$ second limit, while the autonomous additive manufacturing benchmark uses a $3600$ second limit. The global stopping tolerances for all solvers are
\[
\epsilon_{\mathrm{abs}}=0.1,
\qquad
\epsilon_{\mathrm{rel}}=0.01.
\]
Node upper bounds are obtained from feasible evaluations of the true posterior mean using L-BFGS-B. Node lower bounds are obtained by solving the non-convex MIQCP subproblems globally with Gurobi 11.0\cite{gurobi}. The lower bound of MIQCP corresponds to the lower bound of each node. In the multi-threaded configuration, we use $16$ threads for the node MIQCP solves. In this work, we use SCIP 1.14.1, BARON 23.6.23, and maingopy 0.6.0. 

\paragraph{Piecewise-linear envelope settings.}
For the squared exponential kernel experiments reported in the main text, we use a fixed number of line segments for each important kernel term on a node. In our implementation, six line segments gave a good balance between lower-bound quality and MIQCP size. For Mat\'ern kernels, the same construction applies, but if the node interval crosses an inflection point then that point is inserted into the breakpoint set before the segment count is allocated over the concave and convex regions.

\paragraph{Optional preprocessing.}
When root pre-partitioning is enabled, the total number of preprocessing nodes is capped in order to avoid spending excessive time in the initialization stage. In our experiments, we limited the number of preprocessing nodes to at most $500$.

\paragraph{Computing resource}
All experiments are conducted on a single server equipped with an Intel Xeon(R) Gold 6444Y processor and 1 TB of RAM. Except for the multi-core configuration used by PALM-Mean, all solvers are executed on a single CPU core.

\bibliography{references}

\end{document}